\documentclass[a4paper,11pt,reqno,twoside]{amsart}

\usepackage{comment} %added

\usepackage{amssymb}
\usepackage{latexsym}

\theoremstyle{plain}
\newtheorem{propn}{Proposition}[section]
\newtheorem{thm}[propn]{Theorem}
\newtheorem{lemma}[propn]{Lemma}
\newtheorem{cor}[propn]{Corollary}

\theoremstyle{definition}

\theoremstyle{remark}
\newtheorem*{rem}{Remark}
\newtheorem*{rems}{Remarks}
\newtheorem*{eg}{Example}
\newtheorem*{notation}{Notation}

\newcommand{\al}{\alpha}
\newcommand{\be}{\beta}

\newcommand{\Ga}{\Gamma}
\newcommand{\de}{\delta}
\newcommand{\ka}{\kappa}
\newcommand{\la}{\lambda}
\newcommand{\La}{\Lambda}
\newcommand{\vp}{\varpi}
\newcommand{\om}{\omega}
\newcommand{\Om}{\Omega}

\newcommand{\Hil}{\mathsf{H}}
\newcommand{\hil}{\mathsf{h}}
\newcommand{\Kil}{\mathsf{K}}
\newcommand{\kil}{\mathsf{k}}
\newcommand{\Lil}{\mathsf{L}}
\newcommand{\Til}{\mathsf{T}}
\newcommand{\Al}{\mathsf{A}}
\newcommand{\Alg}{\mathcal{A}}

\newcommand{\Cil}{\mathsf{C}}
\newcommand{\Ul}{\mathsf{U}}
\newcommand{\Vil}{\mathsf{V}}
\newcommand{\Wil}{\mathsf{W}}

\newcommand{\uCil}{{}^u \Cil}
\newcommand{\Cpt}{\mathcal{K}}
\newcommand{\ou}{\Ul}
\newcommand{\ov}{\Vil}
\newcommand{\ow}{\Wil}

\newcommand{\init}{\mathfrak{h}}
\newcommand{\noise}{\mathsf{k}}
\newcommand{\khat}{\wh{\noise}}
\newcommand{\Fock}{\mathcal{F}}
\newcommand{\Exps}{\mathcal{E}}

\newcommand{\Step}{\mathbb{S}}

\newcommand{\Real}{\mathbb{R}}
\newcommand{\Rplus}{\Real_+}
\newcommand{\Comp}{\mathbb{C}}
\newcommand{\Nat}{\mathbb{N}}

\newcommand{\ip}[2]{\langle #1, #2 \rangle}

\newcommand{\norm}[1]{\lVert #1 \rVert}
\newcommand{\cbnorm}[1]{\lVert #1 \rVert_\cb}

\newcommand{\w}[1]{\varpi (#1)}
\newcommand{\wxt}{\varpi^\bfx_t}
\newcommand{\wTt}{\varpi^\Til_t}
\newcommand{\bra}[1]{\langle #1 |}
\newcommand{\ket}[1]{| #1 \rangle}
\newcommand{\dyad}[2]{| #1 \rangle \langle #2 |}

\newcommand{\ind}{\mathbf{1}}
\newcommand{\Mat}{\mathrm{M}}
\newcommand{\Diag}{\mathrm{D}}
\newcommand{\flowk}[1]{\mbox{}^{#1} \!k}

\newcommand{\bfu}{\mathbf{u}}

\newcommand{\bfx}{\mathbf{x}}

\newcommand{\bff}{\mathbf{f}}

\newcommand{\cb}{{\text{\tu{cb}}}}
\newcommand{\bd}{{\text{\tu{b}}}}

\newcommand{\wh}{\widehat}
\newcommand{\wt}{\widetilde}
\newcommand{\ol}{\overline}
\newcommand{\ot}{\otimes}
\newcommand{\aot}{\underline{\ot}\,}
\newcommand{\otul}{\underline{\ot}\,}
\newcommand{\vot}{\overline{\ot}\,}
\newcommand{\sot}{\ot}
\newcommand{\otm}{\ot_\Mat}
\newcommand{\mot}{\,{}_\Mat\!\ot}

\newcommand{\op}{\oplus}
\newcommand{\bop}{\bigoplus}
\newcommand{\schur}{\boldsymbol{\cdot}}
\newcommand{\les}{\leqslant}
\newcommand{\ges}{\geqslant}
\newcommand{\To}{\rightarrow}

\newcommand{\Implies}{\Rightarrow}

\newcommand{\ts}{\textstyle}

\newcommand{\ti}{\textit}
\newcommand{\tu}{\textup}

\DeclareMathOperator{\Ran}{Ran}
\DeclareMathOperator{\Lin}{Lin}

\DeclareMathOperator{\id}{id}

\DeclareMathOperator{\im}{Im}
\DeclareMathOperator{\diag}{diag}
\DeclareMathOperator{\ad}{ad}
\DeclareMathOperator{\supp}{supp}

\newenvironment{alist}
{

\begin{enumerate}}
{\end{enumerate}}

\newenvironment{rlist}
{

\begin{enumerate}}
{\end{enumerate}}

\newcommand{\mapassocsemigp}{\mathcal{P}}
\newcommand{\mapassocsemigptwo}{\mathcal{Q}}
\newcommand{\opassocsemigp}{P}

\numberwithin{equation}{section}
\begin{document}

\title[Quantum stochastic semigroups on operator spaces]
{Quantum stochastic cocycles \\
and completely bounded semigroups \\
on operator spaces}
\author{J. Martin Lindsay}
\address{Department of Mathematics and Statistics \\ Lancaster
University \\ Lancaster LA1 4YF \\ UK}
\email{j.m.lindsay@lancs.ac.uk}
\author{Stephen J. Wills}
\address{School of Mathematical Sciences \\ University College
Cork \\ Cork \\ Ireland} \email{s.wills@ucc.ie}
 \date{17.x.2010}

\begin{abstract}
An operator space analysis of quantum stochastic cocycles is
undertaken. These are cocycles with respect to an ampliated CCR
flow, adapted to the associated filtration of subspaces, or
subalgebras. They form a noncommutative analogue of stochastic
semigroups in the sense of Skorohod. One-to-one correspondences are
established between classes of cocycle of interest and corresponding
classes of one-parameter semigroups on associated matrix spaces.
Each of these `global' semigroups may be viewed as the expectation
semigroup of an associated quantum stochastic cocycle on the
corresponding matrix space. The classes of cocycle covered include
completely positive contraction cocycles on an operator system, or
$C^*$-algebra; completely contractive cocycles on an operator space;
and contraction operator cocycles on a Hilbert space. As indicated
by Accardi and Kozyrev, the Schur-action matrix semigroup viewpoint
circumvents technical (domain) limitations inherent in the theory of
quantum stochastic differential equations.
 An infinitesimal analysis of quantum stochastic
cocycles from the present wider perspective is given in a sister paper.
\end{abstract}

\maketitle

\section*{Introduction}
\label{intro}

Cocycles arise in classical and quantum probability theory, for
studying Markov processes and solutions of linear (quantum)
stochastic differential equations (\cite{Skorohod}, \cite{Pin},
\cite{qsc lects}), and in the study of $E_0$-semigroups and product
systems, as a means of perturbing such semigroups (\cite{frock}). In
the former context they are known classically as \emph{stochastic
semigroups}. The core algebraic notion is as follows. Let $\theta =
(\theta_t)_{t \ges 0}$ be a semigroup of linear maps on a vector
space $V$, thus $\theta_0 = \id$ and $\theta_{s+t} = \theta_s \circ
\theta_t$ for all $s,t \ges 0$. Then a family $J = (J_t)_{t \ges 0}$
of linear maps on $V$ is a $\theta$-cocycle if the composed maps
$(J_t \circ \theta_t)_{t \ges 0}$ again form a one-parameter
semigroup, that is
\begin{equation}\label{gen coc}
J_0 = \id, \quad J_{s+t} \circ \theta_{s+t} = J_s \circ \theta_s
\circ J_t \circ \theta_t \quad \text{ for all } s,t \ges 0.
\end{equation}
Further structure is usually imposed, in particular $V$ is often
taken to be an operator algebra $\Al$ and each $\theta_t$ a
${}^*$-homomorphism. For example in the case of $E_0$-semigroups,
$\Al = B(\hil)$ for an infinite-dimensional separable Hilbert space
$\hil$, each $\theta_t$ is unital and ${}^*$-homomorphic, and $J_t$
is given by $J_t = \ad U_t$, where $U$ is a unitary operator (right)
$\theta$-cocycle, that is a family of unitaries $(U_t)_{t \ges 0}$
on $\hil$ satisfying
\begin{equation}\label{gen op coc}
U_0 = I_\hil, \quad U_{s+t} = U_s \theta_s (U_t) \quad \text{ for
all } s,t \ges 0
\end{equation}
(which implies~\eqref{gen coc}). For application to models of
quantum spin systems, $V$ might be a uniformly hyperfinite
$C^*$-algebra with $J$ consisting of completely positive
contractions (\cite{AKphysics}) or *-homomorphisms (\cite{BeW}). For
applications to noncommutative geometry, $V$ would naturally be the
norm-closure of the smooth algebra of a spectral triple with $J$
consisting of *-homomorphisms describing a Brownian motion on the
noncommutative manifold (\cite{SinhaGoswami}). In quantum optics
many examples arise in which $V$ is of the form $B(\hil)$ and $J$ is
given by conjugation by a unitary operator cocycle~\eqref{gen op
coc} (see~\cite{LucBHS}, and references therein).

Note that in~\eqref{gen coc} each $J_t$ need only be defined
on the image of $\theta_t$. Moreover to incorporate concepts such as
Markovianity into this picture $V$ must be equipped with a time
localisation, or filtration, with respect to which the maps $\theta_t$
are adapted; this leads to further simplification in the description
of the maps $J_t$ which we outline next, \emph{without yet being precise
about} the types of tensor products involved. Fix Hilbert spaces
$\init$ and $\noise$. For measurable subsets $I \subset J$ of $\Rplus$
and $t>0$, let $\Fock_I$ denote the symmetric Fock space over $L^2 (I;
\noise)$, and set $\Fock := \Fock_{\Rplus}$. Then we have the
decomposition and inclusion
\[
\init \ot \Fock = \init \ot \Fock_{[0,t[} \ot \Fock_{[t,\infty[},
\quad \Fock_I \cong \Fock_I \ot \w{0_{J \setminus I}} \subset \Fock_J
\]
where $\w{0_{J \setminus I}}$ is the vacuum vector $(1,0,0, \ldots
)$ in $\Fock_{J \setminus I}$. The right shift by $t$ on $L^2
(\Rplus; \noise)$ induces a unitary operator $\init \ot \Fock \To
\init \ot \Fock_{[t,\infty[}$ by second quantisation and ampliation,
allowing the identification of each operator $A \in B(\init \ot
\Fock)$ with an operator on $\init \ot \Fock_{[t,\infty[}$, and
thence an operator $\sigma_t (A)$ on $\init \ot \Fock$ after
ampliating with the identity operator on $\Fock_{[0,t[}$. The
resulting family $\sigma = (\sigma_t)_{t \ges 0}$ is an
$E_0$-semigroup --- the $B(\init)$-ampliated CCR flow of index
$\kil$.

Suppose now that $(j_t)_{t \ges 0}$ is a family of linear maps from a
subalgebra $\Alg$ of $B(\init)$ to $B(\init \ot \Fock)$ satisfying the
initial condition and adaptedness property
\[
j_0 (a) = a \ot I_\Fock,
\quad
j_t (\Alg) \subset \Alg \ot B(\Fock_{[0,t[}) \subset \Alg \ot B(\Fock),
\quad
\text{ for all } t \ges 0.
\]
Let $\wh{\jmath}_t$ denote the tensor product of $j_t$ with the
identity map of $B(\Fock_{[t,\infty[})$, so that $\wh{\jmath}_t: \Alg
\ot B(\Fock_{[t,\infty[}) = \sigma_t (\Alg \ot B(\Fock)) \To \Alg \ot
B(\Fock_{[0,t[}) \ot B(\Fock_{[t,\infty[}) = \Alg \ot B(\Fock)$. Then
(formally, for now) we see that
$(\wh{\jmath}_t)_{t \ges 0}$ is a $\sigma$-cocycle if and only if
\begin{equation}\label{Fock coc}
j_{s+t} = \wh{\jmath}_s \circ \sigma_s \circ j_t \quad \text{ for all
} s,t \ges 0.
\end{equation}
The identity~\eqref{Fock coc} was taken as the \ti{definition} of
 Markovian cocycle in the Fock space setting in~\cite{GCHQ}, and
the subsequent papers \cite{centro}, \cite{aunt},
\cite{version1}. Such a family $(j_t)_{t \ges 0}$ is also a quantum
stochastic process in the sense of \cite{afl}, and the family
$(\mathbb{E}_0 \circ j_t)_{t \ges 0}$, where $\mathbb{E}_0$ is the map
$A \mapsto (\id \ot \omega)(A)$ for the vector state $\omega: T
\mapsto \ip{\vp(0)}{T \vp(0)}$, forms a semigroup on $\Alg$, the
(\ti{vacuum}) \ti{expectation semigroup} of $j$.

The purpose of the papers \cite{GCHQ}, \cite{centro} and \cite{aunt}
is to show that Markovian cocycles can be constructed by solving
quantum stochastic differential equations, and that, under certain
regularity assumptions, all such cocycles arise in this manner. Of
course, one must be precise about the nature of the tensor products
and the definition of the extension $\wh{\jmath}_t$. In \cite{GCHQ}
and \cite{centro} this is achieved by assuming that the algebra $\Alg$
is a von Neumann algebra and that the $j_t$ are unital, injective,
normal ${}^*$-homomorphisms. These ideas were then extended to the
$C^*$-algebraic setting in \cite{aunt} in a way that allows
considerable freedom for the kind of tensor product used, and permits
the study of cocycles for which no assumption about the properties of
each $j_t$ is made --- indeed cocycles for which the families
$(j_t(a))_{t \ges 0}$ consist of unbounded operators naturally arise,
and these could be handled.

The main distinction between \cite{GCHQ} and \cite{aunt} on the one
hand, and \cite{centro} on the other, is the choice of regularity
condition on the expectation semigroup of the cocycle; in the first
two it is assumed to be norm-continuous, which leads to a
\emph{stochastic generator} all of whose components (with respect to
any basis for $\noise$)
  are bounded operators on $\Alg$. In
\cite{centro} the natural assumption of pointwise ultraweak
continuity of the semigroup is made, along with the further
assumption that the domain of the generator contains a
${}^*$-subalgebra as a core. There are examples of interest for
which the second assumption does not hold (\cite{domainCP},
\cite{FrancoDomain}).

For each cocycle $j$ there is a family of associated semigroups
$\{(\mapassocsemigp^{x,y}_t)_{t \ges 0}: x,y \in \kil\}$ on $\Alg$,
which may individually be viewed as expectation semigroups of
perturbations of the cocycle. In turn, these determine the cocycle
through its `semigroup decomposition' (\cite{gran}, \cite{aunt}).
 By a theorem of Parthasarathy-Sunder and Skeide
(see Section~\ref{section qsp}) the cocycle is actually determined
by a small number of associated semigroups---$(1+d)^2$ in case $\dim
\kil = d$ (see Proposition~\ref{J}). Working with completely
positive unital cocycles on a full algebra $\Alg = B(\init)$ and
with one dimension of noise ($\noise = \Comp$), Accardi and Kozyrev
exploit the fact that the four semigroups
$\{(\mapassocsemigp^{x,y}_t)_{t \ges 0}: x,y \in \{0,1\} \subset
\Comp\}$ determine a single semigroup acting componentwise on
$\Mat_2 (\Alg)$, with Schur-action:
\begin{equation}
\label{ext semigp} \mapassocsemigp^\eta_t: [a_{xy}] \mapsto
[\mapassocsemigp^{x,y}_t (a_{xy})],
\end{equation}
which is completely positive and satisfies a certain normalisation
condition arising from unitality of the cocycle (\cite{version1}).
More significantly they show that, from any such semigroup on
$\Mat_2 (\Alg)$, a completely positive unital cocycle on $B(\Alg)$
may be constructed whose associated semigroups include the component
semigroups $(\mapassocsemigp^{x,y})$. This provides a new method for
constructing cocycles circumventing problems with the direct
approach via quantum stochastic calculus which is hampered by
inherent domain constraints on the coefficients of the quantum
stochastic differential equation when these coefficients are
unbounded. Indeed, in the only previous result in this direction
(\cite{unbddgen}), this is met by the severe assumption that a
common domain is left invariant by all coefficients of the QSDE. The
significant advance here is that there is no longer need for the
domains of the components of the stochastic generators to directly
match up, let alone to intersect in a common core. In fact, as we
make clear in \cite{QS+CB2}, for the infinitesimal analysis of QS
cocycles, attention may fruitfully be shifted to the generator of
the \emph{global semigroup}~\eqref{ext semigp}. Potential
applications
 of this viewpoint in the theory of classical SDEs are explored
in~\cite{classprob}.

In this paper we use ideas from the theory of operator spaces to
make a thorough analysis of quantum stochastic cocycles,
establishing and exploiting correspondences between classes of
cocycle of interest and corresponding classes of global semigroups
on associated
 \emph{matrix spaces}
(see Section~\ref{section matrix spaces}).
 Matrix spaces are a hybrid form of operator space
 Fubini product in the sense of Tomiyama
 (\cite{Tom},~\cite{opspbook}) which splice together norm and
 ultraweak topologies (\cite{son},~\cite{LiT}).

 For an operator space $\ov$
in $B(\init_1; \init_2)$, we analyse~\eqref{Fock coc}, for a family
of linear maps $(j_t: \ov \To B(\init_1 \ot \Fock; \init_2 \ot
\Fock))_{t \ges 0}$ on which our \ti{only assumptions} are that they
be completely bounded and map into the $\Fock$-matrix space over
$\ov$. Both assumptions are natural; the first covers all cases of
interest, including that of QS flows on an operator algebra, and QS
contraction cocycles on a Hilbert space; the second is a precise
formulation of the cocycle being ``on $\ov$''. This matrix space
viewpoint, achieved by use of a hybrid topology created through
allowing differing topologies on each tensor factor, permits a
unified treatment of cocycles on a wide variety of structures, all
the time respecting the measure-theoretic nature of the noise. Cases
covered include cocycles on operator spaces, $C^*$/$W^*$-algebras,
operator systems, Hilbert $C^*$-modules, coalgebras, quantum
(semi)groups and Hilbert spaces. The latter class, of bounded
\emph{operator cocycles}, consists of adapted families $X = (X_t)_{t
\ges 0}$ in $B(\init \ot \Fock)$ satisfying~\eqref{gen op coc} for
$\theta = \sigma$; they are in one-to-one correspondence with the
completely bounded cocycles on the column operator space $B(\Comp;
\init)$.

 Matrix spaces also enable the study of cocycles
with arbitrary noise dimension space $\noise$. The global
semigroup~\eqref{ext semigp} is then defined on an
$l^2(\Til)$-matrix space over $\ov$, where $\Til$ might be an
orthonormal basis for $\noise$ augmented by the zero of $\noise$.
These semigroups are incorporated into our approach to quantum
stochastics
([$\text{LW}\!_{2\text{-}4}$]) %preferred for arXiv version
 by being realised as vacuum expectation semigroups of
cocycles (on the corresponding matrix space) obtained by perturbing
matrix-space liftings of the original cocycle
(Proposition~\ref{Wcp}). Crucially this process maintains complete
positivity when the cocycle is on an operator system or
$C^*$-algebra, and encodes contractivity and unitality in
recoverable ways. Accardi and Kozyrev have outlined an extension of
their theory to multidimensional noise in \cite{version1a}.

In the last section (Section~\ref{CCcocs}) we prove a number of
characterisation and reconstruction results extending the one given in
\cite{version1}. We provide results for the following cases:
completely positive contraction cocycles on operator systems and
$C^*$-algebras; completely positive unital cocycles on operator
systems; completely contractive cocycles on operator spaces;
contraction operator cocycles, and positive contraction operator
cocycles, on Hilbert spaces. Liebscher has obtained an alternative
characterisation of uniformly bounded operator quantum stochastic
cocycles (\cite{1particle}).

In this paper we have focused on structural questions for cocycles.
In a sister paper (\cite{QS+CB2}) we analyse the \emph{generation}
 of quantum
stochastic cocycles from the perspective of their global semigroups,
and relate these to characterisations by means of quantum stochastic
differential equations. In particular all issues pertaining to
continuity of the map $t \mapsto j_t$ are deferred to that paper.
The global semigroup viewpoint developed here has already been used
to obtain new results on (operator) quantum stochastic differential
equations (\cite{egs}). However, the essential point, as emphasised
by Accardi and Kozyrev, is that this viewpoint allows one to
overcome inherent limitations of quantum stochastic calculus. This
has been carried out for
 \emph{holomorphic} contraction operator cocycles in \cite{LiS}. In that
paper an infinitesimal characterisation of these cocycles is given,
which manifestly exceeds the scope of quantum stochastic
differential equations.

Finally we mention that families of Schur-action semigroups of
completely positive maps of the type arising in this paper also
appear in the context of product systems of Hilbert modules where
they play a natural role in the analysis of ``Type I'' such systems
(\cite{BBLS}).

\bigskip\noindent
\emph{General notational conventions.} Given a vector-valued
function $f: S \To V$ defined on a set $S$, and subset $A$ of $S$,
we write $f_A$ for the function $\ind_A f: S \to V$ where $\ind_A$
denotes the indicator function of~$A$. All linear spaces are
complex, and Hilbert space inner products are linear in their second
argument and thereby consistent with the standard convention for
Hilbert $C^*$-modules. The $n$-fold orthogonal sum of a Hilbert
space $\hil$ is denoted $\hil^n$. We use the notation $B(X;Y)$ for
the normed space of bounded operators between normed spaces $X$ and
$Y$. For an index set $I$ and a vector space $V$, $\Mat_I (V)$
denotes the linear space of matrices $[a^i_j]_{i,j \in I}$ with
entries from $V$, and $\Mat_n (V)$ the linear space of $n \times n$
matrices with entries from $V$, so that $\Mat_n (V) \cong
\Mat_{\{1,\ldots,n\}} (V) \cong V \otul \Mat_n (\Comp)$, where $
\otul$ denotes the linear tensor product. For a map $\phi: V \To W$
its matrix lifting $\Mat_n(V) \To \Mat_n(W)$, $[a^i_j] \mapsto
[\phi(a^i_j)]$ is denoted $\phi^{(n)}$. The symbol $\ot$ is used for
the tensor product of vectors, Hilbert spaces and Hilbert space
operators; we also use it for spatial/injective tensor products of
operator spaces and completely bounded maps; $\vot$ is used for
ultraweak tensor products of ultraweakly closed operator spaces and
ultraweakly continuous completely bounded maps between such spaces.
 Finally $\de$ is reserved for the Kronecker
symbol, thus $(\de^i)_{i \in I}$ denotes the standard orthonormal
basis of $l^2(I)$.

\section{Operator spaces}
\label{opspaces}

In this section we establish some notation and terminology, introduce
some useful topologies and recall some results from operator space
theory (\cite{Bl+LM}, \cite{opspbook}, \cite{Paulsen}, \cite{Pisier}).
We freely use the abbreviations CB, CP, CC and CI for completely
bounded, completely positive, completely contractive and completely
isometric, and CPC for completely positive and contractive.

\subsection*{Dirac and $E$-notations}

For a Hilbert space $\hil$, we denote $B(\Comp; \hil)$ and $B(\hil;
\Comp)$ by $\ket{\hil}$ and $\bra{\hil}$ respectively, and for $e \in
\hil$ we write $\ket{e} \in \ket{\hil}$ and $\bra{e} \in \bra{\hil}$
for the operators defined by
\[
\ket{e} \la = \la e \ \text{ and } \ \bra{e} f = \ip{e}{f}.
\]
Thus $\dyad{d}{e}$ is the rank one operator in $B(\hil)$ given by $f
\mapsto \ip{e}{f} d$. We usually abbreviate $I_\Hil \ot \ket{e}$ and
$\ket{e} \ot I_\Hil$ to $E_e$, and write $E^e$ for $(E_e)^*$, allowing
context to reveal both the Hilbert space $\Hil$ and the order of the
tensor components.

\subsection*{Matrices of operators}

For Hilbert spaces $\Hil$, $\Kil$ and $\hil$, a choice of orthonormal
basis $\ka = (e_i)_{i \in I}$ for $\hil$ entails isometric
isomorphisms $\Hil \ot \hil \To \bop_{i \in I} \Hil$ and $\Kil \ot
\hil \To \bop_{i \in I} \Kil$ given by $\xi \mapsto {}^\ka \xi :=
(E^{(i)} \xi)_{i \in I}$, and a linear injection
\begin{equation}\label{op to matrix}
B(\Hil \ot \hil; \Kil \ot \hil) \To \Mat_I (B(\Hil; \Kil)), \quad T
\mapsto {}^\ka T:= [E^{(i)} T E_{(j)}]_{i,j \in I}
\end{equation}
where $E_{(j)} := I_\Hil \ot \ket{e_j}: \Hil \To \Hil \ot \hil$ and $E^{(i)} :=
(E_{(i)})^*: \Kil \ot \hil \To \Kil$. We denote by $\Mat_I (B(\Hil;
\Kil))_\bd$ the image of the map~\eqref{op to matrix}; the resulting
matrices act as operators from $\bop_{i \in I} \Hil = \Hil \ot l^2(I)$
to $\bop_{i \in I} \Kil = \Kil \ot l^2(I)$ as matrices should: if $A =
[a^i_j]_{i,j \in I}$ then
\[
Av = \bigl( \ts{\sum_{j \in I} a^i_j v^j} \bigr)_{i \in I}.
\]
Thus the subspace $\Mat_I (B(\Hil; \Kil))_\bd$ is identified with
$B(\bop_{i \in I} \Hil; \bop_{i \in I} \Kil) = B(\Hil \ot l^2(I); \Kil
\ot l^2(I))$, cf.\ the familiar identification
\begin{equation}\label{MnBHK}
\Mat_n (B(\Hil; \Kil)) = B(\Hil^n; \Kil^n) = B(\Hil \ot \Comp^n; \Kil
\ot \Comp^n).
\end{equation}

The following characterisation of nonnegative operators on an
orthogonal direct sum will be useful (\cite{FoF}, \cite{nephews}):
\begin{multline}\label{PosMats}
B(\Hil \op \Kil)_+ = \biggl\{ \begin{bmatrix} A & B \\ B^* &
D \end{bmatrix} : A \in B(\Hil)_+, D \in B(\Kil)_+ \text{ and } B =
A^{1/2} R D^{1/2} \\
\smash{\ \text{ for some contraction } R \in B(\Kil; \Hil) \biggr\}}.
\end{multline}

\subsection*{$\hil$-matrix topologies}
 %\label{Hybrid}

The study of quantum stochastic cocycles on $C^*$-algebras raises
the question of how the topology of the algebra should be spliced with
the measure theoretic noise. This leads naturally to the consideration
of two hybrid topologies which we describe next. The
$\hil$-\ti{ultraweak} (respectively, $\hil$-\ti{weak operator})
topology on $B(\Hil \ot \hil; \Kil \ot\hil)$ is the locally convex
topology defined by the seminorms
\[
p^\om: T \mapsto \norm{\Om (T)}, \: \omega \in B(\hil)_*
\quad
(\text{resp.\ } p_{e,d}: T \mapsto \norm{E^e T E_d}, \: d,e \in \hil),
\]
where $\Om$ is the slice map
\[
\id_{B(\Hil; \Kil)} \vot \: \om: B(\Hil \ot \hil; \Kil \ot \hil) \To
B(\Hil; \Kil)
\]
(\cite{KaR}). Note that $p_{e,d} = p^\om$ for the vector functional
$\om = \om_{e,d}: T \mapsto \ip{e}{Td}$. Thus, in the partial ordering
of topologies on $B(\Hil \ot \hil; \Kil \ot \hil)$,
\begin{center}
$\hil$-W.O.T. \, $\les$ \, $\hil$-ultraweak topology \, $\les$ \, norm
topology.
\end{center}
The $\hil$-ultraweak (respectively $\hil$-weak operator) topology
coincides with the ultraweak (resp.\ weak operator) topology if
$B(\Hil; \Kil)$ is finite dimensional, and coincides with the norm
topology if $\hil$ is finite dimensional. On bounded sets the
$\hil$-ultraweak and $\hil$-w.o. topologies coincide and are, in
general, finer than the ultraweak topology. If $\ov$ and $\ow$ are
subspaces of $B(\Hil; \Kil)$ and $B(\hil)$ respectively, with $\ov$
norm closed, then
\[
\ol{\ov \aot \ow}^{\hil\text{-uw}} \supset \ov \aot
\ol{\ow}^\text{uw} \ \text{ and } \ \ol{\ov \aot
\ow}^{\hil\text{-w.o.}} \supset \ov \aot \ol{\ow}^\text{w.o.}.
\]
Since for each $T \in B(\Hil \ot \hil; \Kil \ot \hil)$ the bounded net
\[
\{(I_\Kil \ot P_F) T (I_\Hil \ot P_F): F \text{ subspace of } \hil,
\dim F < \infty\}
\]
converges $\hil$-weakly to $T$, $B(\Hil; \Kil) \aot B(\hil)$ is
$\hil$-ultraweakly dense in $B(\Hil \ot \hil; \Kil \ot \hil)$.

\subsection*{Operator spaces}

In this paper operator spaces will take concrete form, as closed
subspaces of the space of bounded operators between two Hilbert
spaces, with one class of exceptions. We refer to $\ov$ being an
\ti{operator space in $B(\hil; \kil)$}; $\Mat_n(\ov)$ is thus normed
by being identified with a closed subspace of $B(\hil^n; \kil^n)$
through the identifications~\eqref{MnBHK}. The exception is that, for
concrete operator spaces $\ov$ and $\ow$, the Banach space of
completely bounded operators $CB(\ov; \ow)$ is endowed with matrix
norms making it an operator space via the linear isomorphisms
\begin{equation}\label{CB op space}
\Mat_n (CB(\ov; \ow)) = CB(V; \Mat_n(\ow))
\end{equation}
An operator space $\ov$ in $B(\hil; \kil)$ has a concrete \ti{adjoint
operator space} in $B(\kil; \hil)$:
\[
\ov^\dagger := \{T^*: T \in \ov\}, \quad T \mapsto T^* \
(\text{Hilbert space operator adjoint}).
\]
For concrete operator spaces $\ov$ and $\ow$ and map $\phi: \ov \To
\ow$,
\[
\phi^\dagger:\ov^\dagger \To \ow^\dagger, \quad T^* \mapsto \phi(T)^*,
\]
defines a map which is a completely bounded, respectively completely
isometric, operator if $\phi$ is, moreover $(CB(\ov^\dagger; \ow^\dagger), \phi
\mapsto \phi^\dagger)$ is the adjoint operator space of $CB(\ov; \ow)$
 (see \cite{Bl+LM}). When $\ov^\dagger = \ov$ and $\ow^\dagger = \ow$, we
call a map $\phi: \ov \to \ow$ \emph{real} if it is
adjoint-preserving,that is if it satisfies $\phi^\dagger =
\phi$.

\subsection*{Operator systems}

Recall that a (concrete) \ti{operator system} is a closed subspace
$\ov$ of $B(\hil)$, for some Hilbert space $\hil$, which is closed
under taking adjoints and contains $I_\hil$. Thus an operator system
is linearly generated by its nonnegative elements, and each $\Mat_n
(\ov)$ is itself an operator system in $B(\hil^n)$. Operator systems
also have an abstract characterisation involving the order structure
on the sequence $\{\Mat_n (\ov): n \ges 1\}$, due to Choi and Effros
(\cite{coho}, Theorem~1.2.7).

The following summarises and extends parts of Propositions~2.1 and~3.6
of~\cite{Paulsen}.

\begin{propn}\label{A}
Let $\phi: \ov \To \ow$ be a positive linear map between operator
systems. Then $\phi$ is bounded, real and satisfies $\norm{\phi} \les
2 \norm{\phi(I)}$. Moreover if $\phi$ is $2$-positive then
$\norm{\phi} = \norm{\phi (I)}$, and if $\phi$ is completely positive
then it is completely bounded with $\cbnorm{\phi} = \norm{\phi(I)}$.
\end{propn}

\begin{rem}
In contrast to the situation with $C^*$-algebras, an example of
Arveson (\cite{Paulsen}, Example~2.2) shows that for operator systems
the factor of~$2$ cannot be removed. Note that a completely positive
contraction is completely contractive.
\end{rem}

The following construction, known as Paulsen's $2 \times 2$ matrix
trick, provides a route by which operator system results can be
applied to operator spaces. To each operator space $\ov$ in $B(\hil;
\hil')$ is associated the operator system
\begin{equation}\label{tilde space}
\wt{\ov} = \left\{ \begin{bmatrix} \al I_{\hil'} & a \\ b & \be
I_\hil
\end{bmatrix}: \al, \be \in \Comp, a \in \ov, b\in\ov^\dagger \right\}
\end{equation}
in $B(\hil' \op \hil)$, and each linear map $\phi: \ov \To \ow$,
into an operator space in $B(\Kil; \Kil')$, gives rise to a linear,
real, unital map
\begin{equation*}\label{tilde map}
\wt{\phi}: \wt{\ov} \To \wt{\ow}, \ \
\begin{bmatrix} \al I_{\hil'} & a \\ b & \be
I_\hil \end{bmatrix} \mapsto \begin{bmatrix} \al I_{\Kil'} & \phi(a)
\\ \phi^\dagger (b) & \be I_\Kil \end{bmatrix}.
\end{equation*}

\begin{cor}[\cite{Paulsen}, Lemma~8.1]\label{B}
Let $\phi: \ov \To \ow$ be a linear map between operator spaces. Then
$\phi$ is completely contractive if and only if $\wt{\phi}$ is
completely positive.
\end{cor}

\begin{rem}
Since $\wt{\phi}$ is unital by construction, Proposition~\ref{A}
implies that it is necessarily completely contractive if it is CP.
\end{rem}

The proof is instructive; it starts by noting that the equivalence
of contractivity of $\phi$ and positivity of $\wt{\phi}$ follows
from the characterisation~\eqref{PosMats}; it proceeds by
identifying $(\hil' \op \hil)^n$ with $(\hil')^{n} \op \hil^n$, so
that $\Mat_n(\wt{\ov})$ consists of elements
\begin{equation}\label{a}
\begin{bmatrix} I_{\hil'} \ot \la & A \\ B & I_\hil \ot
\mu \end{bmatrix}
\end{equation}
in which $\la, \mu \in \Mat_n (\Comp)$, $A \in \Mat_n (\ov)$ and $B
\in \Mat_n (\ov^\dagger) = \Mat_n(\ov)^\dagger$; and concludes by
noting that $\wt{\phi}^{(n)}$ corresponds to the extension of
$\wt{\phi^{(n)}}$ which takes~\eqref{a} to
\[
\begin{bmatrix} I_{\Kil'} \ot \la & \phi^{(n)} (A) \\
\phi^{(n)\dagger}(B) & I_\Kil \ot \mu \end{bmatrix},
\]
and that this map is positive if $\wt{\phi^{(n)}}$ is.

\subsection*{Nonunital $C^*$-algebras}

The prototypical example of an operator system is a unital
$C^*$-algebra, acting nondegenerately on a Hilbert space $\hil$.
Unfortunately the operator system axioms/definition exclude
nonunital $C^*$-algebras (for extensions of the Choi-Effros
characterisation in this direction see~\cite{matord}). The following
result, which is standard (e.g.\ \cite{coho}, Theorem~1.2.1), is
collected here for convenience.

\begin{propn}\label{Ua}
Let $\Cil$ be a nonunital $C^*$-algebra acting nondegenerately on a
Hilbert space $\hil$, and let $1$ denote the identity of $B(\hil)$,
then the unital $C^*$-algebra
\begin{equation*}\label{uA}
\uCil := C^*(\Cil \cup \{1\}) = \{a + z1: a \in \Cil, z \in \Comp\}
\end{equation*}
contains $\Cil$ as a maximal ideal. If $\phi: \Cil \To B(\Kil)$ is a
linear completely positive map, for some Hilbert space $\Kil$, then
$\norm{\phi} = \sup \{\norm{\phi(a)}: a \in \Cil_{+,1}\}$ and for any
$C \ges \norm{\phi}$,
\[
\psi: \uCil \To B(\Kil), \quad a + z1 \mapsto \phi (a) + zC I_\Kil
\]
defines a CP extension of $\phi$ satisfying $\norm{\psi} = C$.
\end{propn}

\subsection*{From operators to CB maps}

The next result provides the basic mechanism whereby operator cocycles
will be viewed as completely bounded mapping cocycles.

\begin{propn}
\label{basic mechanism}
For Hilbert spaces $\hil$, $\kil$, $\Hil$ and $\Kil$, the identity
\begin{equation}\label{between}
\phi (\ket{u}) = X (\ket{u} \ot I_\Hil) \quad u \in \hil,
\end{equation}
establishes a completely isometric isomorphism between the spaces $B
(\hil \ot \Hil; \Kil)$ and $CB (\ket{\hil}; B (\Hil;
\Kil))$. Similarly, $\psi (\bra{u}) = (\bra{u} \ot I_\Hil) Y$
establishes a completely isometric isomorphism between $B (\Hil; \kil
\ot \Kil)$ and $CB (\bra{\kil}; B (\Hil; \Kil))$.
\end{propn}

\begin{proof}
Set $\ov = B(\hil \ot \Hil; \Kil)$ and $\ow = CB (\ket{\hil}; B(\Hil;
\Kil))$.

Let $\mathcal{T}: \ket{\hil} \To B(\Hil; \hil \ot \Hil)$ denote the
ampliation $\ket{u} \mapsto \ket{u} \ot I_\Hil$, and for $X \in \ov$
let $L_X$ denote the corresponding left multiplication operator
$B(\Hil; \hil \ot \Hil) \To B(\Hil; \Kil)$. Then $\mathcal{T}$ is CI,
and $L_X$ is CB with $\cbnorm{L_X} = \norm{X}$. Thus $X \mapsto \phi_X
:= L_X \circ \mathcal{T}$ defines a contraction $\Phi = \Phi_{\hil,
\Hil; \Kil}: \ov \To \ow$, with $\phi_X$ satisfying~\eqref{between}.

For $\phi \in \ow$ let $X^0_\phi: \hil \aot \Hil \To \Kil$ be the
linearisation of the bilinear map $(u, f) \mapsto \phi (\ket{u}) f$.
Then, for $\xi \in \hil \aot \Hil$, expressing $\xi$ in the form
$\sum_{i=1}^n u_i \ot f^i$ where $u_1, \ldots ,u_n$ are mutually
orthogonal unit vectors in $\hil$,
\[
\norm{X^0_\phi \xi} = \norm{\phi^{(n)} (T) \bff} \ \text{ where } T=
\begin{bmatrix} \ket{u_1} & \cdots & \ket{u_n} \\ 0 & \cdots & 0 \\
\vdots & \ddots & \\ 0 & & 0 \end{bmatrix} \text{ and } \bff =
\begin{pmatrix} f^1 \\ \vdots \\ f^n \end{pmatrix}.
\]
Since $T \in B (\Comp^n; \hil^n)$ has norm one and $\norm{\bff} =
\norm{\xi}$, $X^0_\phi$ is bounded with norm at most $\cbnorm{\phi}$;
let $X_\phi \in B(\hil \ot \Hil; \Kil)$ be its continuous
extension. Then $X_\phi$ satisfies~\eqref{between} and $\phi \mapsto
X_\phi$ defines a contraction $\Psi: \ow \To \ov$.

Clearly $\Phi$ and $\Psi$ are mutually inverse, thus they are
isometric too, and so are Banach space isometric isomorphisms.

Now let $X = [X^i_j] \in \Mat_n (B(\hil \ot \Hil; \Kil)) = B (\hil \ot
\Hil^n; \Kil^n)$. Then, under the identification $\Mat_n
(CB(\ket{\hil}; B(\Hil; \Kil))) = CB(\ket{\hil}; B(\Hil^n; \Kil^n))$
(see~\eqref{CB op space})
\[
[X^i_j (\ket{u} \ot I_\Hil)] = X (\ket{u} \ot I_{\Hil^n}) = \phi_X
(\ket{u}),
\]
and so $(\Phi_{\hil, \Hil; \Kil})^{(n)} = \Phi_{\hil, \Hil^n;
\Kil^n}$. It follows that $\Phi$ is a CI isomorphism.

The second isomorphism is implemented by $(\Phi_{\kil, \Kil;
\Hil})^\dagger$.
\end{proof}

The next result details some traffic in the above correspondence.

\begin{cor}
 \label{map v. op inj}
For Hilbert spaces $\hil$, $\Hil$ and $\Kil$, let
 $X\in B(\hil \ot \Hil; \Kil)$ and
 $\phi\in CB(\ket{\hil}; B(\Hil; \Kil))$
 correspond according to~\eqref{between}.
\begin{alist}
\item
 If $X$ is injective then $\phi$ is injective.
\item\label{IsoImpCI}
 If $X$ is isometric then $\phi$ is completely isometric.
\end{alist}
\noindent
The converses hold when $\Hil = \Comp$.
\end{cor}

\begin{proof}
Identifying $\Lil^n$ with $\Comp^n \ot \Lil$ for the Hilbert spaces
$\Lil =\Hil,\Kil$, (a) and (b) follow from the identity
\[
\phi^{(n)} (T) = (I_{\Comp^n} \ot X) (T \ot I_\Hil), \quad T \in
\Mat_n(\ket{\hil}), n\in\Nat.
\]
When $\Hil = \Comp$, $B(\hil; \Kil) \cong CB(\ket{\hil}; \ket{\Kil})$,
with $\phi(\ket{u}) = \ket{Xu}$, from which the converses follow
immediately.
\end{proof}

\begin{eg}
Let $\Kil = \hil \ot \Hil'$ and $X = I_\hil \ot R$, where $R \in
B(\Hil; \Hil')$ is noninjective and of norm one, for some Hilbert
space $\Hil'$. Then $X$ is noninjective but, for each $n\in\Nat$,
\[
\phi_X^{(n)} (T) = T \ot
R, \quad T \in \Mat_n (\ket{\hil}),
\]
 so $\phi_X$ is completely
isometric.
\end{eg}

\section{Matrix spaces}
\label{section matrix spaces}

In this section we describe an abstract matrix construction over a (concrete)
operator space, with Hilbert space as index set, that was introduced
in~\cite{son}, and we develop some of its properties. The matrix
spaces considered here are a coordinate-free version of the spaces of
infinite matrices over an operator space introduced by Effros and Ruan
(see~\cite{opspbook}, Chapter~10). Maps between spaces of matrices
that are comprised of a matrix of linear maps, each one mapping
between corresponding components of the matrices (i.e.\ having
\emph{Schur action}), are then characterised.

\begin{lemma}\label{slice}
For an operator space $\ov$ in $B(\Hil; \Kil)$ and total subsets
$\Til$ and $\Til'$ of a Hilbert space $\hil$, we have
\begin{multline}
 \label{h matrix}
\{T \in B(\Hil \ot \hil; \Kil \ot \hil): E^x T E_y \in \ov \text{ for
all } x \in \Til', y \in \Til\} \\
= \{T \in B(\Hil \ot \hil; \Kil \ot \hil): (\id_{B(\Hil; \Kil)} \vot
\: \om) (T) \in \ov \text{ for all } \om \in B(\hil)_*\}.
\end{multline}
\end{lemma}

\begin{proof}
 This follows from the following facts. For
 $T\in B(\Hil \ot \hil; \Kil \ot \hil)$,
 $E^x T E_y = (\id_{B(\Hil; \Kil)} \vot \: \om_{x,y}) (T)$, the set
 $\{\om \in B(\hil)_*: (\id_{B(\Hil; \Kil)} \vot \: \om) (T)
\in \ov\}$ is a norm closed subspace of $B(\hil)_*$, and the set
$\{\om_{x,y}: x \in \Til', y \in \Til\}$ is total in $B(\hil)_*$.
\end{proof}

The (\ti{right}) \ti{$\hil$-matrix space over $\ov$} is the operator
space~\eqref{h matrix}; we denote it $\ov \otm
B(\hil)$. [Previous notation, used in~\cite{son} and elsewhere: $\Mat
(\hil; \ov)_\bd$.]

\begin{rems}
Rectangular matrix spaces $\ov \otm B(\hil; \hil')$ and left matrix
spaces $B(\hil) \mot \ov$ are defined in the obvious way. If $\ka =
(e_i)_{i \in I}$ is an orthonormal basis for $\hil$ then, in the
notation~\eqref{op to matrix}, we have the identification
\begin{equation}\label{AbsToMat}
T \in \ov \otm B(\hil) \longleftrightarrow {}^\ka T \in \Mat_I (V)_\bd
\end{equation}
where $\Mat_I (\ov)_\bd$ is defined to be the operator space $\Mat_I
(B(\Hil; \Kil))_\bd \cap \Mat_I (\ov)$. Note that $\Mat_I (\ov)_\bd$
has a description valid for abstract matrix spaces $\ov$, namely
\[
\{A \in \Mat_I (\ov): \sup_{\La \subset\subset I} \norm{A^{[\La]}} <
\infty\}
\]
where $A^{[\La]} \in \Mat_\La(\ov)$ denotes the (finite) submatrix of
$A$ obtained by cut-off (\cite{opspbook}, Chapter 10). Significantly,
$\ov \otm B(\hil)$ does too, namely $CB(B(\hil)_*;\ov)$ (\cite{LiT}).
\end{rems}

The following is Lemma~1.1 of~\cite{son}, adapted for square-matrix
spaces:

\begin{lemma}\label{C}
Let $\ov$ be an operator space in $B(\Hil; \Kil)$, and $\hil$ and
$\kil$ any pair of Hilbert spaces. The natural associativity map
$B(\Hil \ot \hil; \Kil \ot \hil) \vot B(\kil) \To B(\Hil; \Kil) \vot
B(\hil \ot \kil)$ restricts to a completely isometric isomorphism
\[
(\ov \otm B(\hil)) \otm B(\kil) = \ov \otm B(\hil \ot \kil).
\]
\end{lemma}

\subsection*{Diagonal matrix spaces}

Let $\ov$ be an operator space and $\hil$ a Hilbert space with
orthonormal basis $\ka = (e_i)_{i \in I}$. We define the operator
space
\[
\ov \otm \Diag_\ka (\hil) := \{T \in \ov \otm B(\hil): E^{(i)} T
E_{(j)} = 0 \text{ for } i \neq j\}
\]
(with $E^{(i)}$ and $E_{(j)}$ as in~\eqref{op to matrix}) and refer to
it as the \ti{$\ka$-diagonal $\hil$-matrix space over $\ov$}. Clearly
the completely isometric isomorphism~\eqref{AbsToMat} restricts to a
completely isometric isomorphism
\[
\ov \otm \Diag_\ka (\hil) \To \Diag_I (\ov)_\bd
\]
where the \ti{diagonal operator subspace of} $\Mat_I
(\ov)_\bd$ is defined in the obvious way:
\[
\Diag_I (\ov)_\bd := \{[a^i_j] \in \Mat_I (\ov)_\bd: a^i_j = 0 \text{
for } i \neq j\}.
\]

\subsection*{Matrix space liftings}

A feature of the matrix space construction that was exploited
in~\cite{son} is that completely bounded maps between operator spaces
induce completely bounded maps between corresponding matrix
spaces. This is detailed in the next result, whose proof follows the
same lines as that of Lemma 1.2 of~\cite{son}. The result shows in
particular that, as operator spaces, matrix spaces do not depend on
the concrete representation of the underlying operator space
(as already remarked above). In
brief, any completely isometric isomorphism between (concrete)
operator spaces $\ov$ and $\ow$ induces a completely isometric
isomorphism between $\ov \otm B(\hil)$ and $\ow \otm B(\hil)$.

\begin{lemma}\label{D}
For a completely bounded operator between operator spaces $\phi: \ov
\To \ow$, and a Hilbert space $\hil$, there is a unique map $\phi \otm
\id_{B(\hil)}: \ov \otm B(\hil) \To \ow \otm B(\hil)$ satisfying
\[
E^e (\phi \otm \id_{B(\hil)}) (T) E_d = \phi (E^e T E_d)
\quad
\text{ for all } d,e \in \hil, T \in \ov \otm B(\hil);
\]
it is linear and completely bounded, moreover \tu{(}unless $\hil =
\{0\}$\tu{)} it satisfies $\cbnorm{\phi \otm \id_{B(\hil)}} =
\cbnorm{\phi} = \norm{\phi \otm \id_{B(l^2)}}$.
\end{lemma}

Clearly these maps are the coordinate-free counterparts to the
sequence of induced maps $(\phi^{(n)})_{n \ges 1}$. Indeed $\phi
\otm \id_{B(\hil)}$ is often abbreviated to $\phi^{(\hil)}$.

\begin{rem}
If the operator space $\ow$ is of the form $\ou \otm B(\Hil)$ then we
write
\begin{equation}\label{flipped lift}
\phi^\hil \ \text{ for } \ \Sigma \circ \phi^{(\hil)}: \ov \otm
B(\hil) \To \ou \otm B(\hil \ot \Hil),
\end{equation}
$\Sigma$ being the tensor flip $\ou \otm B(\Hil \ot \hil) \To \ou \otm
B(\hil \ot \Hil)$.
\end{rem}

We next address topological questions concerning matrix spaces and
induced maps between matrix spaces.
 For $\om \in B(\hil)_*$, we have
 \[
(\id_{B(\Hil; \Kil)} \vot \om)(\ov\otm B(\hil)) \subset \ov;
\]
 the induced operator $\ov\otm B(\hil) \to \ov$ is denoted
 $\id_\ov \otm\ \om$.
Let $\phi \in CB(\ov; \ow)$ for operator spaces $\ov$ and $\ow$.
 The collection of functionals $\om \in B(\hil)_*$
satisfying
\begin{equation}\label{normalstates}
(\id_\ow \otm \ \om) \circ (\phi \otm \id_{B(\hil)}) = \phi \circ
(\id_\ov \otm \ \om)
\end{equation}
is norm-closed and contains the norm-total family $\{\om_{e,d}: d,e
\in \hil\}$, and so~\eqref{normalstates} holds for all $\om \in
B(\hil)_*$. Accordingly we denote the resulting map $\ov \otm B(\hil)
\To \ow$ by $\phi \otm \om$.

\begin{rem}
 There is a sense in which a version of identity~\eqref{normalstates}
 holding for all $\phi$ actually characterises the normality of
 $\omega$. Precise sufficient conditions are given in~\cite{Tom},
 Theorem 5.1; see also~\cite{Neu}, Theorem 5.4.

 \end{rem}

 The following result is now evident.

\begin{lemma}\label{h topology}
Let $\ov$ and $\ow$ be operator spaces.
\begin{alist}
\item
$\ol{\ov \aot B(\hil)}^{\hil\mathrm{-w.o.}} = \ov \otm B(\hil) = \ol{\ov
\aot B(\hil)}^{\hil\mathrm{-uw}}.$
\item
 For a completely bounded operator $\phi: \ov \To \ow$, $\phi
\otm \id_{B(\hil)}$ is both $\hil$-ultraweakly continuous and
$\hil$-weak operator continuous.
\end{alist}
\end{lemma}

Note the further relations
\[
\ov \sot B(\hil) \subset \ov \otm B(\hil) \subset \ol{\ov}^\text{uw}
\vot B(\hil);
\]
the first inclusion being an equality if either $\ov$ or $\hil$ is
finite dimensional and the second being an equality if and only if
$\ov$ is ultraweakly closed. Thus the $\hil$-matrix space over an
ultraweakly closed operator space is its ultraweak tensor product
with~$B(\hil)$.

\subsection*{Maps with Schur-action}

For operator spaces $\ov$ and $\ow$, index set $I$ and linear map
$\phi: \Mat_I (\ov)_\bd \To \Mat_I (\ow)_\bd$, we can define maps
$\phi^i_j: \ov \To \ow$ by
\begin{equation}\label{cpt maps}
\phi^i_j (a) = E^{(i)} \phi(E_{(i)} a E^{(j)}) E_{(j)}.
\end{equation}
We say that $\phi$ has \ti{Schur-action} if these maps determine
$\phi$, in the sense that $\phi$ acts componentwise through
\[
\phi ([a^i_j]) = [\phi^i_j (a^i_j)].
\]
For example if $\varphi \in CB(\ov; \ow)$ and $\hil = l^2(I)$, then
the map $\varphi \otm \id_{B(\hil)}$ from Lemma~\ref{D}, viewed as a
map $\Mat_I (\ov)_\bd \To \Mat_I (\ow)_\bd$ (through~\eqref{AbsToMat})
has Schur-action: $\varphi^{(\hil)} ([a^i_j]) = [\varphi(a^i_j)]$.

For operator spaces $\ov$ and $\ow$ and Hilbert space $\hil$ with
orthonormal basis $\ka =(e_i)_{i \in I}$, a linear map $\phi: \ov \otm
B(\hil) \To \ow \otm B(\hil)$ will be called $\ka$-\ti{decomposable}
if the induced map ${}^\ka\phi: \Mat_I (\ov)_\bd \To \Mat_I (\ow)_\bd$
given by ${}^\ka\phi ({}^\ka T) := {}^\ka(\phi(T))$ (see~\eqref{op to
matrix}), has Schur-action.

We next establish criteria for a map to have Schur-action. To this end
consider the orthogonal projections
\[
p_k = [\de^i_k \de^k_j I_\Hil] =
E_{(k)}E^{(k)},
 \quad k\in I,
\]
  in $\Mat_I (B(\Hil))_\bd$
 (for any Hilbert space
$\Hil$). The following lemma is easily verified.

\begin{lemma}\label{E}
Let $\phi$ be a linear map $\Mat_I (\ov)_\bd \To \Mat_I (\ow)_\bd$ for
operator spaces $\ov$ and $\ow$, and index set $I$. Then the following
are equivalent\tu{:}
\begin{rlist}
\item
$\phi$ has Schur-action.
\item
$\phi (p_i A p_j) = p_i \phi (A) p_j \ \text{ for all } A \in
\Mat_I (\ov)_\bd, i,j \in I$.
\end{rlist}
\end{lemma}

Specialising to completely positive unital maps between operator
systems, Schur-action has the following useful characterisation.

\begin{propn}\label{F}
Let $\ov$ and $\ow$ be operator systems, $I$ an index set, and $\phi$
a linear map $\Mat_I (\ov)_\bd \To \Mat_I (\ow)_\bd$.
\begin{alist}
\item
If $\phi$ is unital and has Schur-action then $\phi(p_i) = p_i$ for
all $i \in I$.
\item
If $\phi(p_i) = p_i$ for all $i \in I$ and $\phi$ is a completely
positive contraction then $\phi$ is unital and has Schur-action.
\end{alist}
\end{propn}

\begin{proof}
(a) Immediate from Lemma~\ref{E}.

(b) Let $\hil = l^2(I)$ and identify $\Mat_I (\ov)_\bd$ with $\ov \otm
B(\hil)$. Suppose that $\ov$ and $\ow$ are operator systems in
$B(\Hil)$ and $B(\Kil)$ respectively, then, by Arveson's Hahn-Banach
Theorem (\cite{Paulsen}, Theorem~7.5), $\phi$ extends to a CP
contraction $\phi': B(\Hil \ot \hil) \To B(\Kil \ot \hil)$. Moreover,
$\phi'(p_i^2) = \phi'(p_i) = p_i = \phi'(p_i)^2$, so we have equality
in the Kadison-Schwarz inequality and hence, by a result of Choi
(\cite{Paulsen}, Proposition~3.18), it follows that
\[
\phi' (p_i T p_j) = p_i \phi' (T) p_j \quad \text{ for all } T \in
B(\Hil \ot \hil), i,j \in I.
\]
In particular for $T \in \ov \otm B(\hil)$, $p_i T p_j \in \ov \otm
B(\hil)$ and so $\phi$ has Schur-action by Lemma~\ref{E}. Finally,
$\phi$ is unital since for all $i,j \in I$
\[
p_i \phi(I) p_j = \phi(p_ip_j) = \delta^i_j \phi(p_i) = \delta^i_j p_i
= p_i I p_j. \qed
\]
\renewcommand{\qed}{}
\end{proof}

\begin{rems}
(i) Choi's result is stated in~\cite{Paulsen} for unital maps, but the
proof shows that contractivity suffices.

(ii) In the absence of a normality assumption on $\phi$, the following
example shows that the contractivity assumption in~(b) is needed. Let
$\hil = l^2(\Nat)$, let $\Cpt$ denote the algebra of compact operators
on $\hil$ and let $\ov = \ow = \Comp$, so that $\Mat_\Nat (\ov)_\bd
\cong B(\hil)$, and let $\varphi = \id_\Cpt$. For $\la > 1$ define a
CP extension of $\varphi$ to $\Cpt +\Comp I$ by
\[
\varphi^\la (T +\mu I) = T +\la \mu I,
\]
using Proposition~\ref{Ua}, and extend this to a CP map $\Phi^\la$ on
$B(\hil)$ by Arveson's Hahn-Banach Theorem. Then $\Phi^\la (p_i) =
p_i$ but $\Phi^\la (I) = \la I \neq I$.
\end{rems}

\section{Quantum stochastic processes}
\label{section qsp}

Let $\noise$ be a fixed but arbitrary complex Hilbert space, referred
to at the \ti{noise dimension space}. The orthogonal sum $\Comp \op
\noise$ is denoted $\khat$.

\subsection*{Fock space}

The symmetric Fock space over $L^2 (\Rplus; \noise)$ is denoted
$\Fock$. We use normalised exponential vectors (also called coherent
vectors)
\[
\w{f} := e^{-\frac{1}{2} \norm{f}^2} \bigl( (n!)^{-1/2} f^{\ot n}
\bigr)_{n \ges 0}, \ f \in L^2 (\Rplus; \noise);
\]
these are linearly independent and satisfy $\ip{\w{f}}{\w{g}} =
\exp(-\chi(f,g))$, where, for any Hilbert space $\Hil$,
\[
\chi(u,v) := \tfrac{1}{2}(\norm{u}^2 +\norm{v}^2) -\ip{u}{v}, \quad
u,v \in \Hil.
\]
For a subset $\Til$ of $\noise$ we set
\[
\Exps_\Til := \Lin \{ \w{f}: f \in \Step_\Til\},
\]
where $\Step_\Til$ denotes the collection of (right-continuous)
$\Til$-valued step functions in $L^2 (\Rplus; \noise)$, and abbreviate
to $\Exps$ and $\Step$ when $\Til$ is all of $\noise$ (note that
necessarily $0 \in \Til$). Then $\Exps_\Til$ is dense in $\Fock$ if
and only if $\Til$ is total in $\noise$. For example $\Til = \{0\}
\cup \{d_i: i \in I_0\}$ for an orthonormal basis $\{d_i\}_{i \in
I_0}$ for $\noise$. For a proof of this result, which is due to
Parthasarathy and Sunder, and Skeide, and the basics of quantum
stochastics, we refer to~\cite{qsc lects}.

For each subinterval $J$ of $\Rplus$ we denote the symmetric Fock
space over $L^2 (J; \Rplus)$ by $\Fock_J$, and the identity operator
on $\Fock_J$ by $I_J$. Then the natural identifications
\[
\Fock = \Fock_{[0,r[} \ot \Fock_{[r,t[} \ot \Fock_{[t,\infty[}, \quad 0
\les r \les t \les \infty,
\]
are effected by
\[
\w{f} \mapsto \w{f|_{[0,r[}} \ot \w{f|_{[r,t[}} \ot
\w{f|_{[t,\infty[}},
\]
and $\Fock_{[0,t[}$ is naturally isometric to the subspace
$\Fock_{[0,t[} \ot \vp(0|_{[t,\infty[})$ of $\Fock$.

By this means we make the identifications
\[
B(\Fock) = B(\Fock_{[0,r[}) \vot B(\Fock_{[r,t[}) \vot
B(\Fock_{[t,\infty[})
\]
and, in turn,
\[
B(\Fock_{[r,t[}) = I_{[0,r[} \ot B(\Fock_{[r,t[}) \ot I_{[t,\infty[}
\subset B(\Fock).
\]

The \ti{Fock-Weyl operators} may be defined on $\Fock$ by continuous
linear extension of the prescription
\begin{equation}\label{Weyl ops}
W(f): \w{g} \mapsto e^{-i \im \ip{f}{g}} \w{f+g}, \quad f,g \in
L^2(\Rplus; \noise).
\end{equation}
These are unitary operators satisfying the canonical commutation
relations in exponential/Weyl form:
\[
W(f) W(g) = e^{-i \im \ip{f}{g}} W(f+g).
\]
Note that the probabilistic normalisation (\cite{qsc lects},
\cite{Meyer}, \cite{Partha}) is preferred here rather than the usual
quantum theoretic one (\cite{BrR}).

We next record a notation which will be heavily used in the sequel:
\begin{equation}\label{Ef notation}
E(f) := E_{\w{f}}, \quad f \in L^2 (\Rplus; \noise).
\end{equation}
Thus $E(f) = (I \ot W(f)) E_{\w{0}}$.

The \ti{CCR flow} on $B(\Fock)$ is the one-parameter semigroup of
normal, unital, ${}^*$-endomorphisms determined by
\[
\sigma^\noise_t \bigl(W(f)\bigr) = W(s_t f)
\]
where $(s_t)_{t \ges 0}$ is the semigroup of right shifts on $L^2
(\Rplus; \noise)$. Thus $\sigma^\noise_t (B(\Fock)) = I_{[0,t[} \ot
B(\Fock_{[t,\infty[})$.

\subsection*{Processes}

Let $\ov$ be an operator space in $B(\init; \init')$. A \ti{bounded
quantum stochastic process in $\ov$} (with noise dimension space
$\noise$) is a family of operators $(X_t)_{t \ges 0}$ satisfying the
adaptedness condition
\[
X_t \in \ov \otm B(\Fock_{[0,t[}) \ot I_{[t,\infty[} \ \text{ for all
} t \in \Rplus.
\]
In practice a weak measurability condition is also imposed; however in
this paper such an assumption plays no role. When $\ov = B(\init)$ we
speak of a \ti{bounded QS process on $\init$}. The self-adjoint
unitary process $R^\noise$ defined by
\[
R^\noise_t \vp(f) = \vp(r_tf) \ \text{ where } \ (r_t f)(s)
= \begin{cases} f(t-s) & \text{if } s\in [0,t[, \\ f(s) & \text{if }
s\in [t,\infty[, \end{cases}
\]
plays a fundamental role. For a bounded process $X$ on $\init$ the
\emph{time-reversed} process is defined by
\begin{equation}\label{time-reversed}
X^R := (R_t X_t R_t)_{t\geq 0} \ \text{ where } \ R_t = I_\init \ot
R_t^\noise.
\end{equation}

We are primarily interested in bounded QS processes \ti{on} an
operator space $\ov$. These are families of bounded operators
$k_t: \ov \To \ov \otm B(\Fock)$ such that $(k_t(a))_{t \ges 0}$ is a
process in $\ov$ for each $a \in \ov$. Such a process is called
\ti{completely bounded}, \ti{completely contractive}, or (when $\ov$
is an operator system or $C^*$-algebra) \ti{completely positive}, if
each $k_t$ is. The property $k_t (\ov) \subset \ov \otm B(\Fock)$ is a
noncommutative form of \ti{Feller condition} (see~\cite{son}).

\section{Completely bounded quantum stochastic cocycles}
\label{opspacecoc}

Let $\ov$ be an operator space in $B(\init; \init')$. A completely
bounded QS process $k$ on $\ov$ is called a \ti{quantum stochastic
cocycle on $\ov$} if it satisfies
\[
k_0 = \iota_\Fock \ \text{ and } \ k_{r+t} = \wh{k}_r \circ \sigma_r
\circ k_t \ \text{ for } r,t \in \Rplus,
\]
where $\iota_\Fock$ denotes the ampliation $a \mapsto a \ot I_\Fock$,
$\sigma_r$ is the shift obtained by restriction to $\ov \otm B(\Fock)$
of the map $\id_{B(\init; \init')} \vot \sigma^\noise_r$, and
$\wh{k}_r := k_r \otm \id_{B(\Fock_{[r,\infty[})}$. For the lifting
$\wh{k}_r$, the following identifications are invoked:
\[
\Ran \sigma_r = \ov \otm B(\Fock_{[r,\infty[}) \ \text{ and } \ \ov
\otm B(\Fock_{[0,r[}) \otm B(\Fock_{[r,\infty[}) = \ov \otm B(\Fock).
\]
In this paper, \emph{all QS cocycles on an operator space will be
assumed to be completely bounded}.

To each bounded process $k$ we associate the family of bounded
operators $k^{f,g}_t: \ov \To \ov$, indexed by ordered pairs $(f,g)$
from $L^2_\text{loc} (\Rplus; \kil)$, defined, in the
notation~\eqref{Ef notation}, by
\begin{equation}\label{reduced maps}
k^{f,g}_t (a) = E(f_{[0,t[})^* k_t (a) E(g_{[0,t[}).
\end{equation}

\begin{rems}
(i) Unnormalised exponential vectors were used for defining the maps
$k^{f,g}_t $ in earlier papers
([$\text{LW}\!_{\text{2,3}}$]) %preferred for arXiv version
% (\cite{aunt}, \cite{son}]). for published version
 Benefits of
normalising will be seen later.

(ii) We identify the noise dimension space $\noise$ with the constant
functions in $L^2_\text{loc} (\Rplus; \noise)$.
\end{rems}

In the context of operator algebras and completely positive processes,
 the proposition below appeared in~\cite{gran} for finite-dimensional $\noise$,
and in~\cite{aunt}.

\begin{propn}\label{J}
Let $k$ be a completely bounded process on an operator space $\ov$
 in $B(\init; \init')$,
set $\mapassocsemigp^{x,y}_t := k^{x,y}_t$ \tu{(}$x,y \in \noise$,
$t \ges 0$\tu{)} and let $\Til$ and $\Til'$ be total subsets of
$\noise$ containing $0$. Then the following are equivalent\tu{:}
\begin{rlist}
\item
$k$ is a QS cocycle on $\ov$\tu{;}
\item
$k_0^{f,g} = \id_\ov$ and $k^{f,g}_{r+t} = k^{f,g}_r \circ k^{s^*_r
f, s^*_r g}_t \ \text{ for all } f \in \Step_\Til$, $g \in
\Step_{\Til'}$ and $r, t \ges 0$.
\item
 For all $x \in \Til$ and $y\in\Til'$, $(\mapassocsemigp^{x,y}_t)_{t
\ges 0}$ defines a semigroup on $\ov$, and for all $f \in
\Step_\Til$, $g \in \Step_{\Til'}$ and $t > 0$
\begin{equation}\label{k semigroup rep}
k^{f,g}_t = \mapassocsemigp^{x_0,y_0}_{t_1-t_0} \circ \cdots \circ
\mapassocsemigp^{x_m,y_m}_{t_{m+1}-t_m},
\end{equation}
where $t_0 = 0$, $t_{m+1}=t$, $\{t_1 < \cdots < t_m \}$ is the
\tu{(}possibly empty\tu{)} union of the sets of points of
discontinuity of $f$ and $g$ in $]0,t[$ and, for $i=0,\ldots,m$,
$x_i := f(t_i)$ and $y_i := g(t_i)$.
\item
 For all $f \in \Step_\Til$, $g \in \Step_{\Til'}$ and $t \ges 0$,
$k_0^{f,g} = \id_\ov$ and, whenever $\{0 = s_0 \les s_1 \les \ldots
\les s_{n+1} = t\}$ contains all the points of discontinuity of
$f_{[0,t[}$ and $g_{[0,t[}$,
\begin{equation}\label{k semigroup rep2}
k^{f,g}_t = \mapassocsemigp^{x_0,y_0}_{s_1-s_0} \circ \cdots \circ
\mapassocsemigp^{x_n,y_n}_{s_{n+1}-s_n}
\end{equation}
where, for $j=0,\ldots,n$, $x_j:= f(s_j)$ and $y_j:= g(s_j)$.
\end{rlist}
\end{propn}

\begin{proof}
Let $r,s \in
\Rplus$ with $r \les s$. The following identities, in which $h \in L^2
(\Rplus; \kil)$, $T \in \ov \otm B(\Fock_{[r,\infty[})$ and $X \in
B(\init; \init') \vot B(\Fock)$, are straightforward consequences of
the definitions:
\begin{align*}
&E (h) = E (h_{[r,\infty[}) E(h_{[0,r]}), \\
&E (f_{[r,\infty[})^* \wh{k}_r (T) E (g_{[r,\infty[})= k_r \bigl(
E(f_{[r,\infty[})^* T E(g_{[r,\infty[}) \bigr), \\
&E(f_{[r,\infty[})^* \sigma_r (X) E(g_{[r,\infty[}) = E(s_r^* f)^* X
E(s^*_r g),
\end{align*}
as is the inclusion
\[
\wh{k}_r \bigl(\ov \otm B(\Fock_{[r,s[}) \ot I_{[s,\infty[} \bigr)
\subset \ov \otm B(\Fock_{[0,s[}) \ot I_{[s,\infty[}.
\]
Therefore, for $a \in V$ and $r,t \in \Rplus$,
\begin{multline*}
E(f_{[0,r+t[})^* (\wh{k}_r \circ \sigma_r \circ k_t) (a)
E(g_{[0,r+t[}) \\
\begin{aligned}
&= E(f_{[0,r[})^* E(f_{[r,r+t[}) \wh{k}_r \bigl( (\sigma_r \circ
k_t)(a) \bigr) E(g_{[r,r+t[}) E(g_{[0,r[}) \\
&= E(f_{[0,r[})^* k_r \bigl( E(s_r^* f_{[0,r+t[}) k_t (a) E(s^*_r
g_{[0,r+t[}) \bigr) E(g_{[0,r[}) \\
&= k^{f,g}_r \circ k^{s^*_r f, s^*_r g}_t (a),
\end{aligned}
\end{multline*}
since $(\wh{k}_r \circ \sigma_r \circ k_t) (a) \in \ov \otm
B(\Fock_{[0,r+t[})$, and so the equivalence of (i) and (ii) follows
from the totality of $\Exps_\Til$ and $\Exps_{\Til'}$ in $\Fock$. The
equivalence of (ii), (iii) and (iv) follows from the fact that $s^*_u
z = z$ for all $z \in \kil$ and $u \in \Rplus$.
\end{proof}

We refer to $\{\mapassocsemigp^{x,y}: x,y \in \noise\}$ as the
\emph{associated semigroups} of the cocycle, $\mapassocsemigp^{0,0}$
as its \emph{\tu{(}vacuum\tu{)} expectation semigroup}, and~\eqref{k
semigroup rep}
 or ~\eqref{k semigroup rep2} as the
 \emph{semigroup decomposition/characterisation}
for QS cocycles. Note that if the cocycle $k$ is contractive then so
are each of its associated semigroups. The following is an immediate
consequence of the above.

\begin{cor}\label{cor: uniqueness}
Let $j$ and $k$ be completely bounded QS cocycles on an operator space,
 with respective associated semigroups
$\{\mapassocsemigp^{x,y}: x,y \in \noise\}$ and
$\{\mapassocsemigptwo^{x,y}: x,y \in \noise\}$, and let $\Til$ and
$\Til'$ be total subsets of $\noise$ containing $0$. Then $j = k$ if
and only if $\mapassocsemigp^{x,y} = \mapassocsemigptwo^{x,y}$ for
all $x \in \Til'$ and $y \in \Til$.
\end{cor}

\begin{notation}
If $\{\mathcal{R}^{x,y}: x,y \in \Til\}$ is an indexed family of
linear maps on the operator space $\ov$ then, for $n \ges 1$ and
$\bfx \in \Til^n$, we write $\mathcal{R}^\bfx$ for the Schur-action
map on $\Mat_n (\ov)$ with component maps $\{\mathcal{R}^{x_i,x_j}:
i,j = 1, \ldots, n\}$:
\[
\mathcal{R}^{\bfx}([u_j^i]) = [\mathcal{R}^{x_i,x_j}(u_j^i)]
\]
\end{notation}

In conjunction with the representation~\eqref{knN} below, the
following result is useful for extracting positivity and contractivity
properties of QS cocycles.

\begin{propn}\label{kft maps}
Let $k$ be a QS cocycle on an operator space $\ov$
 with associated semigroups
$\{\mapassocsemigp^{x,y}: x,y \in \noise\}$,
  and let $\bff \in
\Step^N$. Then $k^\bff_t$, the Schur-action map on $\Mat_N (\ov)$ with
components $\{k^{f_i,f_j}_t: i,j=1,\ldots,N\}$, satisfies
\begin{equation}\label{kf semigroup rep}
k^\bff_t = \mapassocsemigp^{\bfx (0)}_{t_1-t_0} \circ \cdots \circ
\mapassocsemigp^{\bfx(n)}_{t_{n+1}-t_n}, \text{ where } \bfx(k) :=
\bff(t_k),
\end{equation}
whenever $\{0 = t_0 \les \cdots \les t_{n+1} = t\}$ contains the
discontinuities of $\bff_{[0,t[}$.
\end{propn}

\begin{proof}
Since each semigroup $\mapassocsemigp^{\bfx (k)}$ has Schur-action,
the result follows immediately from Proposition~\ref{J}.
\end{proof}

\subsection*{Matrices and liftings of cocycles}

For $i,j \in \{1,\ldots,n\}$ let $\ov^i_j$ be an operator space in
$B(\init_i; \init_j')$ and let $k^i_j$ be a QS cocycle on
$\ov^i_j$. Then the Schur-action
\[
k_t ([a^i_j]) := [k^i_j (t) (a^i_j)]
\]
defines a QS cocycle on the operator space
\[
\ov := \{ [a^i_j] \in B(\ts{\bop_i \init_i; \bop_j \init_j')}): a^i_j
\in \ov^i_j \ \text{ for all } i,j\}.
\]
This follows from the identity
\[
k_t^{f,g} ([a^i_j]) = [k^i_j (t)^{f,g} (a^i_j)],
\]
for $f,g \in L^2_\text{loc} (\Rplus; \kil)$.

As an example let $k$ be a QS cocycle on an operator space $\ov$ in
$B(\init; \init')$, then $k^\dagger$ is a cocycle on $\ov^\dagger$
with $(k^\dagger_t)^{f,g} = (k^{g,f}_t)^\dagger$ and so,
recalling~\eqref{CB op space} and~\eqref{tilde space}, $\wt{k}$ is a
cocycle on the operator system $\wt{\ov}$. The associated semigroups
of $\wt{k}$ are given by
\begin{equation}\label{t}
\wt{\mapassocsemigp}^{x,y}_t \left(\begin{bmatrix} \al I_{\init'} &
a \\ b & \be I_\init \end{bmatrix} \right) = \begin{bmatrix} \al
e^{-t \chi(x,y)} I_{\init'} & \mapassocsemigp^{x,y}_t (a) \\
(\mapassocsemigp^{y,x}_t)^\dagger (b) & \be e^{-t \chi(x,y)} I_\init
\end{bmatrix};
\end{equation}
where $\{\mapassocsemigp^{x,y}: x,y\in\noise\}$ are the associated
semigroups of $k$. The identification of $\Mat_n (\wt{\ov})$ given
in~\eqref{a} entails the action
\begin{equation}\label{u}
\wt{\mapassocsemigp}^\bfx_t: \begin{bmatrix} I_{\init'} \ot \la & A
\\ B & I_\init \ot \mu \end{bmatrix} \mapsto \begin{bmatrix}
I_{\init'} \ot (\la \schur \wxt) & \mapassocsemigp^\bfx_t (A) \\
(\mapassocsemigp^\bfx_t)^\dagger (B) & I_\init \ot (\mu \schur \wxt)
\end{bmatrix},
\end{equation}
where for each $n \ges 1$, $\bfx \in \noise^n$ and $t \ges 0$, we
define the Grammian matrix
\begin{equation}\label{GramMat}
\wxt := \bigl[ \ip{\w{x^i_{[0,t[}}}{\w{x^j_{[0,t[}}} \bigr] =
[e^{-t\chi (x^i,x^j)}] \in \Mat_n (\Comp)_+,
\end{equation}
and $\schur$ denotes the Schur product of scalar matrices.

\medskip\noindent
\emph{Warning.} The scope of the tilde is important: in general
\[
\begin{bmatrix}
\al e^{-t\chi(x,y)} I_{\init'} & \mapassocsemigp^{x,y}_t(a) \\
(\mapassocsemigp^{y,x}_t)^\dagger (b) & \be e^{-t\chi(x,y)} I_\init
\end{bmatrix}
\neq
\begin{bmatrix}
\al I_{\init'} & \mapassocsemigp^{x,y}_t (a) \\
\mapassocsemigp^{x,y}_t (b^*)^* & \be I_\init
\end{bmatrix},
\]
so $\wt{\mapassocsemigp}^{x,y}_t \neq \wt{\mapassocsemigp^{x,y}_t}$.

The following identity is useful for the examination of properties of
a QS cocycle $k$ on an operator space $\ov$ in $B(\init; \init')$. If
$\xi \in (\init \aot \Exps)^n$ with representation $\xi^i = \sum^N_{p
=1} u^i_p \ot \w{f_p^i}, i=1, \ldots, n$, and $\xi' \in (\init' \aot
\Exps)^n$ with corresponding `primed' representation then, for each $A
\in \Mat_n (\ov)$,
\begin{equation}\label{knN}
\ip{\xi'}{k^{(n)}_t (A) \xi} = \langle\eta', k^{(nN)}_t (A \ot
\square_N) \eta\rangle
\end{equation}
where
\begin{equation}\label{box}
\square_N := \begin{bmatrix} 1 & 1 & \cdots & 1 \\ 1 & 1 & \cdots & 1
\\ \vdots & \vdots & \ddots & \vdots \\ 1 & 1 & \cdots &
1 \end{bmatrix} \in \Mat_N (\Comp) \ \text{ and } \ \eta
= \begin{pmatrix} \eta_1 \\ \vdots \\ \eta_N \end{pmatrix} \in \bigl(
(\init \aot \Exps)^n \bigr)^N
\end{equation}
with $\eta^i_p = u^i_p \ot \w{f_p^i}$, and $\eta'$ defined
correspondingly.

Another useful construction of new cocycles from old is obtained by
lifting, as follows. Recall the notation~\eqref{flipped lift}. If $k$
is a QS cocycle on an operator space $\ov$ then, for any Hilbert space
$\hil $, $k^{\hil} := (k^{\hil}_t)_{t \ges 0}$ defines a QS cocycle on
$\ov \otm B(\hil)$. As an immediate application we extend a
fundamental estimate for $C_0$-semigroups to QS cocycles.

\begin{propn}\label{Jb}
Let $k$ be a QS cocycle on an operator space $\ov$, with locally
bounded CB norm. Then there exist constants $M \ges 1$ and $\omega \in
\Real$ such that
\[
\cbnorm{k_t} \les Me^{\omega t} \quad \text{ for all } t \ges 0.
\]
\end{propn}

\begin{proof}
Let $\hil$ be any infinite dimensional Hilbert space. Then by the
cocycle identity for $k^{\hil}$, and the complete isometry of the
shifts,
\[
\cbnorm{k_{r+t}} =\norm{k^{\hil}_{r+t}} \les \norm{\wh{k^{\hil}_r}}
\norm{k^{\hil}_t} = \cbnorm{k_r} \cbnorm{k_t}
\]
The result therefore follows by standard arguments from semigroup
theory (e.g.\ \cite{Davies}, Lemma~1.2.1).
\end{proof}

\subsection*{Operator QS cocycles}

In \cite{gran} and
 [$\text{LW}\!_{1,2}$], %preferred for arXiv version
% for submission version:
% \cite{mother} and \cite{aunt}
 properties of an
operator process on $\init$ are deduced from results concerning QS
cocycles, flows and QS differential equations on operator algebras,
by constructing an associated process on $B(\init)$. However, full
algebras are not necessarily the best choice, as is shown in
Theorem~\ref{W} below. More significantly the tools of operator
space theory --- in particular Proposition~\ref{basic mechanism} ---
provide an alternative means of seeing the two types of process from
a common viewpoint facilitating further analysis.

A bounded process $X$ on $\init $ is a \ti{left} (resp.\ \ti{right})
\ti{quantum stochastic cocycle} if $X_0 = I_{\init \ot \Fock}$ and
\[
X_{r+t} = X_r \sigma_r (X_t) \quad (\text{resp.\ } X_{r+t} = \sigma_r
(X_t) X_r) \quad \text{ for all } r,t\ges 0.
\]
Thus $X$ is a left cocycle on $\init $ if and only if $X^* :=
(X^*_s)_{s \ges 0}$ is a right cocycle on $\init$.

\begin{propn}\label{op to map cocycles}
For bounded QS processes $X$ and $Y$ on $\init$ define completely
bounded QS processes $\flowk{1}$ on $\ket{\init}$, $\flowk{2}$ on
$\bra{\init}$, and $\flowk{3}$ and $\flowk{4}$ on $B(\init)$ by
\begin{align*}
\flowk{1}_s (\ket{u}) &= X_s (\ket{u} \ot I_\Fock), &
\flowk{3}_s (a) &= X_s (a \ot I_\Fock) X^*_s, \ \text{ and } \\
\flowk{2}_s (\bra{u}) &= (\bra{u} \ot I_\Fock) Y_s, &
\flowk{4}_s (a) &= Y^*_s (a \ot I_\Fock) Y_s,
\end{align*}
for $u \in \init$ and $a \in B(\init)$.
\begin{alist}
\item
$\flowk{1}$ is a cocycle on $\ket{\init}$ if and only if $X$ is a left
cocycle on $\init$, in which case $\flowk{3}$ is a cocycle on
$B(\init)$.
\item
$\flowk{2}$ is a cocycle on $\bra{\init}$ if and only if $Y$ is a
right cocycle on $\init$, in which case $\flowk{4}$ is a cocycle on
$B(\init)$.
\end{alist}
Furthermore, $\norm{X_s} = \cbnorm{\flowk{1}_s} = \norm{\flowk{3}_s} =
\cbnorm{\flowk{3}_s}$\tu{;} similarly for $Y$, $\flowk{2}$ and
$\flowk{4}$.
\end{propn}

\begin{proof}
Adaptedness of the process $X$ amounts to the statement: for all $t
\ges 0$, $v \in \init$ and $f \in L^2([0,t[; \noise)$ there is $\xi
\in \init \ot \Fock_{[0,t[}$ such that
\[
X_t \bigl(v \ot \w{f} \ot \w{g}\bigr) = \xi \ot \w{g} \quad (g \in
L^2([t,\infty[; \noise)).
\]
Similarly for $\flowk{1}$ and $\flowk{3}$ with the left hand side
replaced by $\flowk{1}_t (\ket{v}) \w{f}$ or $\flowk{3}_t (a) v \ot
\w{f}$. The implications in~(a) then follow from the easily verified
relations
\begin{align*}
&\wh{\flowk{1}}_s \circ \sigma_s \circ \flowk{1}_t (\ket{u})
= X_s \sigma_s (X_t) (\ket{u} \ot I_\Fock); \\
&\wh{\flowk{3}}_r \circ \sigma_r \circ \flowk{3}_t (a) = X_r \sigma_r
(X_t) (a \ot I_\Fock) \sigma_s (X_t)^* X^*_r,
\end{align*}
and (b) follows by taking adjoints, since $\flowk{1}_s^\dagger =
\flowk{2}_s$ if $Y = X^*$. The norm identities are immediate.
\end{proof}

Thus the natural completely isometric isomorphism between $B(\init)
\vot B(\Fock)$ and $CB(\ket{\init}; \ket{\init} \vot B(\Fock))$
(resp.\ $CB(\bra{\init}; \bra{\init} \vot B(\Fock))$) interchanges
processes on $\init $ with processes on $\ket{\init}$ (resp.\
$\bra{\init}$) and left (resp.\ right) QS cocycles on $\init$ with QS
cocycles on $\ket{\init}$ (resp.\ $\bra{\init}$). Note also the
relations
\begin{equation}\label{k12 semigps}
\flowk{1}^{f,g}_s (\ket{u}) = X^{f,g}_s \ket{u} \ \text{ and } \
\flowk{2}^{f,g}_s (\bra{u}) = \bra{u} Y^{f,g}_s,
\end{equation}
where $X^{f,g}_s$ is defined analogously to~\eqref{reduced maps}.
These identities and Proposition~\ref{J} yields the following:

\begin{propn}
Let $X$ be a bounded QS process on $\init$, set
$\opassocsemigp^{x,y}_t := X^{x,y}_t$ \tu{(}$x, y \in \noise$, $t
\ges 0$\tu{)}, and let $\Til$ and $\Til'$ be total subsets of
$\noise$ containing $0$. Then the following are equivalent\tu{:}
\begin{rlist}
\item
$X$ is a left \tu{(}respectively right\tu{)} QS cocycle on $\init$.
\item
$X^{f,g}_0 = I_\init$ and $X^{f,g}_{r+t} = X^{f,g}_r X^{s^*_r f,
s^*_r g}_t$ \tu{(}resp.\ $X^{f,g}_{r+t} = X^{s^*_r f, s^*_r g}_t
X^{f,g}_r$\tu{)} for all $f \in \Step_\Til$, $g \in \Step_{\Til'}$
and $r, t \ges 0$.
\item
 For all $x \in \Til$ and $y \in \Til'$, $(\opassocsemigp^{x,y}_t)_{t
\ges 0}$ defines a semigroup on $\init$, and for all $f \in
\Step_\Til$, $g \in \Step_{\Til'}$ and $t > 0$
\begin{equation}\label{Opj}
X^{f,g}_t = \opassocsemigp^{x_0,y_0}_{t_1-t_0} \cdots
\opassocsemigp^{x_m,y_m}_{t_{m+1}-t_m} \quad (\text{resp.\ } \
\opassocsemigp^{x_m,y_m}_{t_{m+1}-t_m} \cdots
\opassocsemigp^{x_0,y_0}_{t_1-t_0})
\end{equation}
where $t_0=0$, $t_{m+1}=t$, $\{t_1 < \ldots < t_m\}$ is the
\tu{(}possibly empty\tu{)} union of the sets of discontinuity of $f$
and $g$ in $]0,t[$ and, for $i=0, \cdots, n$, $x_i := f(t_i)$ and
$y_i = g(t_i)$.
\item
 For all $f \in \Step_\Til$, $g \in \Step_{\Til'}$ and $t \ges 0$,
$X^{f,g}_0 = I$ and, whenever $\{0 = s_0 \les s_1 \les \ldots \les
s_{n+1} = t\}$ contains all the points of discontinuity of
$f_{[0,t[}$ and $g_{[0,t[}$,
\[
X^{f,g}_t = \opassocsemigp^{x_0,y_0}_{s_1-s_0} \cdots
\opassocsemigp^{x_n,y_n}_{s_{n+1}-s_n} \quad (\text{resp.\ } \
\opassocsemigp^{x_n,y_n}_{s_{n+1}-s_n} \cdots
\opassocsemigp^{x_0,y_0}_{s_1-s_0})
\]
where, for $j=0, \cdots, n$, $x_j := f(s_j)$ and $y_j := g(s_j)$.
\end{rlist}
\end{propn}

The adjoint operation on bounded QS processes on a Hilbert space
exchanges left and right QS cocycles. As an immediate corollary of the
above semigroup decomposition/characterisation it follow that the
time-reversal procedure defined in~\eqref{time-reversed} does too.

\begin{cor}\label{cor time-reversal}
Let $X$ be a bounded QS process on $\init$. Then $X$ is a left
\tu{(}respectively right\tu{)} QS cocycle if and only if the
time-reversed process $X^R$ is a right \tu{(}resp.\ left\tu{)}
cocycle.
\end{cor}

\begin{rem}
For a bounded left (respectively, right) QS cocycle $X$ on $\init$,
\[
\wt{X} := (X^R)^* = (X^*)^R
\]
defines another bounded left (resp.\ right) QS cocycle on $\init$,
known as the \ti{dual cocycle} (\cite{Jou}).
\end{rem}

\subsection*{Cocycle dichotomies}

We next give some dichotomies which the cocycle laws entail.

\begin{propn}\label{dichotomy}
Let $k$ be a QS cocycle on $\ov$, let $j$ be a
 QS cocycle on an operator system $\ow$, and let $X$ be a
\tu{(}left or right, bounded\tu{)} operator QS cocycle on $\init$. Then
each of the following sets is either empty or all of
$]0,\infty[$\tu{:}
\begin{alist}
\item
$\{t > 0: k_t \text{ is injective}\}$\tu{;}
\item
$\{t > 0: X_t \text{ is injective}\}$\tu{;}
\item
assuming that $k$ is completely contractive, \\
$\{t > 0: k_t \text{ is completely isometric}\}$\tu{;}
\item
 assuming that $j$ is completely positive and contractive, \\
$\{t > 0: j_t \text{ is unital}\}$\tu{;}
\item
assuming that $X$ is a contraction cocycle,
\begin{alist}
\item
$\{t > 0: X_t \text{ is isometric}\}$\tu{;}
\item
$\{t > 0: X_t \text{ is coisometric}\}$.
\end{alist}
\end{alist}
\end{propn}

\begin{proof}
For $0 \les r \les t \les u$ where $t>0$, and $l = k$ or $j$,
\begin{align}
l_t &= \wh{l}_{t-r} \circ \sigma_{t-r} \circ l_r \label{Adichotomy} \\
&= \wh{l}_r \circ \sigma_r \circ l_{t-r}, \text{ and}
\label{AAdichotomy} \\
l_u &= (\wh{l}_t \circ \sigma_t)^{\circ N} \circ l_{u-Nt},
\label{Bdichotomy}
\end{align}
where $N:= [t^{-1}u]$, so that $0 \les (u-Nt) \les t$. Let $t > 0$.
If $k_t$ is injective then~\eqref{Adichotomy} implies that $k_r$ is
injective for all $r\les t$. If $k$ is CC and $k_t$ is CI
then~\eqref{Adichotomy} implies that
\[
\norm{A} = \norm{k^{(n)}_t (A)} \les \norm{k^{(n)}_r (A)} \les
\norm{A} \quad \text{ for all } n \in \Nat, A \in \Mat_n (\ov), r \les
t
\]
so $k_r$ is CI for all $r \les t$. If $j_t$ is unital
then~\eqref{AAdichotomy} implies that
\[
I = j_t(I) \les \wh{\jmath}_r(I) \les I \quad \text{ for all } r \les
t
\]
so $j_r$ is unital for all $r \les t$. Since injectivity or complete
isometry for $k_t$ (respectively, unitality for $j_t$) implies the
same property for $(\wh{k}_t \circ \sigma_t)$ (resp.\ $(\wh{\jmath}_t
\circ \sigma_t)$), parts (a), (c) and (d) now follow
from~\eqref{Bdichotomy}.

Since $X$ is injective/isometric/coisometric if and only if the
time-reversed process $X^R$ is, and $X$ is isometric (respectively
coisometric) if and only if the adjoint process $X^*$ is coisometric
(resp.\ isometric) it suffices (by Corollary~\ref{cor time-reversal})
to prove~(b) and (e\,ii) in the case that $X$ is a left
cocycle. That~(b) holds follows by a similar argument to parts~(a),
(c) and~(d), noting that for any $R \in B(\init \ot \Fock)$,
\[
\{t \ges 0: \sigma_t(R) \text{ is injective}\} = \emptyset \text{ or }
\Rplus.
\]
Finally, for (e\,ii), let $j_t(B) := X_t (B \ot I_\Fock) X^*_t$, a CPC
cocycle on $B(\init)$ by Proposition~\ref{op to map cocycles}, so that
the result follows from~(d) since $X$ is coisometric if and only if
$j$ is unital.
\end{proof}

\subsection*{Associated $\Ga$-cocycle and global $\Ga$-semigroup}

The associated semigroups $\mapassocsemigp^{x,y}$ and
$\opassocsemigp^{x,y}$ were studied individually in \cite{aunt}, and
when they have bounded generators they can be used to construct a
stochastic generator for the cocycle. Here, following Accardi and
Kozyrev, we show that the entire family can usefully be treated as a
single object.

Recall the Fock-Weyl operators defined in~\eqref{Weyl ops}. Let $\Ga$
be a map $I \To \noise$, for some (index) set $I$. Then, writing
$(\de^\al)_{\al \in I}$ for the usual basis for $l^2 (I)$,
\begin{equation}\label{Weyl matrix}
W^\Ga_t: \de^\al \ot \xi \mapsto \de^\al \ot W \bigl(
\Ga(\al)_{[0,t[} \bigr) \xi, \quad \al \in I, \xi \in \Fock,
\end{equation}
determines a unitary QS process $W^\Ga$ on $l^2(I)$. If $B(l^2(I) \ot
\Fock)$ is identified with $\Mat_I (B(\Fock))_\bd$ then
\[
W^\Ga_t = [\de^\al_\be W \bigl( \Ga(\al)_{[0,t[} \bigr)] \in \Diag_I
(B(\Fock))_\bd.
\]

We consider such processes in three guises. For a total subset $\Til$
of $\noise$ containing $0$, set $W^\Til = W^\Ga$ where $\Ga$ is the
inclusion map $\Til \To \noise$. For an orthonormal basis $(d_i)_{i
\in I_0}$ for $\noise$, letting
\begin{equation}\label{d to e basis}
d_0 = 0 \text{ in } \noise, \quad e_0 = \begin{pmatrix} 1 \\
0 \end{pmatrix} \text{ and } e_i = \begin{pmatrix} 0 \\
d_i \end{pmatrix} \text{ in } \khat, \quad \text{ and } I = I_0 \cup
\{0\},
\end{equation}
so that $\eta = (e_\al)_{\al \in I}$ is an orthonormal basis for
$\khat$, set $W^\eta = W^\Ga$ where $\Ga$ is the map $I \To \noise$,
$\al \To d_\al$. Finally, for $n \ges 1$ and $\bfx \in \Til^n$, set
$W^\bfx = W^\Ga$ where $\Ga$ is the map $\{1,\ldots,n\} \To \noise$,
$i \mapsto x^i$. Thus $W^\Til$ is a process on $l^2(\Til)$; $W^\eta$
is a process on $\khat \cong l^2(I)$ consisting of $\eta$-diagonal
operators:
\[
W^\eta_t \in \Diag_\eta (\khat)_\bd \mot B(\Fock);
\]
and $W^\bfx$ is the process on $\Comp^n$ given by
\begin{equation}
 \label{Wx mat}
 W_t^{\mathbf{x}} :=
\begin{bmatrix} W(x^1_{[0,t[}) & & \\ & \ddots & \\ & &
W(x^n_{[0,t[}) \end{bmatrix} \in \Mat_n (\Comp) \ot B(\Fock).
\end{equation}
The proof of the following is straightforward.

\begin{propn}\label{Wcp}
Let $W^\Ga$ be the unitary QS process on $l^2(I)$ associated with a
map $\Ga:I \To \noise$, as in~\eqref{Weyl matrix}, and let $k$ be a
QS cocycle on an operator space $\ov$ with associated semigroups
$\{\mapassocsemigp^{x,y}: x,y \in \noise\}$.
\begin{alist}
\item
$W^\Ga$ is a strongly continuous left QS cocycle.
\item
$W^\Ga$ is also a right QS cocycle on $l^2(I)$.
\item
Each of the associated semigroups of the cocycle $W^\Ga$ is norm
continuous if and only if the function $\Ga$ is bounded.
\item
The completely bounded QS process on $\ov \otm B(l^2(I)) = \Mat_I
(\ov)_\bd$ defined by
\begin{equation}\label{PertLift}
K^\Ga_t := (I_{\init'} \ot W^\Ga_t)^* k^{l^2(I)}_t (\cdot) (I_\init
\ot W^\Ga_t)
\end{equation}
is a QS cocycle whose expectation semigroup is the Schur-action
semigroup given by
\begin{equation}\label{P Gamma}
\mapassocsemigp^\Ga_t := \bigl[
\mapassocsemigp^{\Ga(\al),\Ga(\be)}_t \bigr]_{\al,\be \in I}.
\end{equation}
\end{alist}
\end{propn}

\begin{rems}
(i) Since the cocycle $W^\Ga$ is unitary, and thus contractive, norm
continuity of every associated semigroup follows from norm continuity
of its expectation semigroup~(\cite{aunt}).

(ii) The proposition applies to $W^\Til$ for a subset $\Til$ of
$\noise$, or $W^\eta$ for an orthonormal basis $\eta$ of $\khat$
associated with some basis of $\noise$ as in~\eqref{d to e basis}.
Thus for $\eta = (e_\al)_{\al \in I}$, $K^\eta$ defines a QS cocycle
on $\ov \otm B(\khat)$ whose associated expectation semigroup
$\mapassocsemigp^\eta$ is the semigroup of $\eta$-decomposable maps
with component semigroups $\mapassocsemigp^{(\al,\be)} :=
\mapassocsemigp^{d_\al,d_\be}$, $\al, \be \in I$.

(iii) In part (d) we may take a second function $\Gamma': I' \to
\noise$ and, by using a rectangular lifting of $k$, obtain a
completely bounded QS cocycle on the operator space $\ov \otm
B(l^2(I');l^2(I)) = \Mat_{I,I'} (\ov)_\bd$, with Schur-action expectation
semigroup
\[
\mapassocsemigp^{\Ga,\Ga'} := \Bigl( \bigl[
\mapassocsemigp^{\Ga(\al),\Ga'(\al')}_t \bigr]_{\al \in I, \al' \in
\Ga'} \Bigr)_{t \ges 0}.
\]
\end{rems}

The semigroups in the collection $\{\mapassocsemigp^\bfx: \bfx \in
\bigcup_{n\in \Nat} \Til^n \}$ appearing above and in the following
section are also seen to be the vacuum expectation semigroups of QS
cocycles associated to $k$ through~\eqref{PertLift}
 and~\eqref{P Gamma}.

\section{Characterisation and reconstruction}
\label{CCcocs}

In this section we focus our attention initially on completely
positive contraction cocycles $k$ on an operator system. We derive
additional properties satisfied by the family of associated
semigroups, and, following Accardi and Kozyrev, show that conversely,
for any such family of semigroups on an operator system indexed by a
total subset $\Til$ of $\kil$ containing $0$, there is a cocycle $k$
for which this is its family of associated semigroups. We then apply
Paulsen's $2 \times 2$ matrix trick to extend this characterisation of
completely positive contraction cocycles on an operator system to
completely contractive cocycles on any operator space. Finally a
characterisation of contraction operator cocycles is obtained. Along
the way we derive characterisations of completely positive contraction
cocycles on a $C^*$-algebra and positive contraction operator cocycles
on a Hilbert space.

We will repeatedly make use of Schur products beyond the context of
scalar matrices (as used in the previous section), but only in the
favourable circumstances guaranteed by the following elementary
observation.

\begin{lemma}\label{pos S prod}
Let $\ov$ be an operator system, or $C^*$-algebra, in $B(\hil)$, and
let $A \in \Mat_n(\ov)_+$ and $\la \in \Mat_n(\Comp)_+$ for some $n
\in \Nat$. Then $A \schur \la := [a^i_j \la^i_j] \in \Mat_n(\ov)_+$.
\end{lemma}

\begin{proof}
Let $\bfu = (u^i) \in \hil^n$. Then, setting $\mu = \la^{1/2}$ and
$\bfu_{(k)} = (\mu^k_i u^i)_{i=1}^n \in \hil^n$ ($k=1, \ldots, n$),
we have
\[
\ip{\bfu}{(A \schur \la) \bfu} = \sum_k \ip{\bfu_{(k)}}{A \bfu_{(k)}}
\ges 0.\qed
\]
\renewcommand{\qed}{}
\end{proof}

\begin{rems}
(i) Using Schur isometries, a version of this result for infinite
matrices is exploited in~\cite{QS+CB2}.

(ii) For a commutative $C^*$-algebra $\Al$, the Schur product of
positive elements of $\Mat_n (\Al)_+$ is easily seen to be positive by
identifying $\Mat_n (\Al)_+$ with $C(\Sigma; \Mat_n (\Comp)_+)$, where
$\Sigma$ is the spectrum of $\Al$.
\end{rems}

\subsection*{Completely positive contraction cocycles}

Recall the notations $\wxt$ and $\square_n$ introduced
in~\eqref{GramMat} and~\eqref{box}.

\begin{propn}\label{PP}
Let $k$ be a QS cocycle on an operator space $\ov$, let $\Til$ be a
total subset of $\kil$ containing $0$, and consider the family of
semigroups $\{\mapassocsemigp^\bfx: \bfx \in \bigcup_{n\ges 1}
\Til^n\}$ defined from the associated semigroups of $k$.
\begin{alist}
\item
 If $\ov$ is an operator system, or $C^*$-algebra, then
 the following are equivalent\tu{:} \label{CPeq}
\begin{alist}
\item
The cocycle $k$ is completely positive.
\item
Each semigroup $\mapassocsemigp^\bfx$ is completely positive.
\item
Each semigroup $\mapassocsemigp^\bfx$ is positive.
\end{alist}
\item\label{CPCeq}
If $\ov$ is an operator system and $k$ is completely positive then
the following are equivalent\tu{:}
\begin{alist}
\item
$k$ is a contraction cocycle.
\item
$\mapassocsemigp^\bfx_t (I_\init \ot \la) \les I_\init \ot (\wxt
\schur \la)$ for all $n \ges 1$, $\bfx \in \Til^n$, $\la \in \Mat_n
(\Comp)_+$ and $t \ges 0$.
\item
$\mapassocsemigp^\bfx_t (I_\init \ot \square_n) \les I_\init \ot
\wxt$ for all $n \ges 1$, $\bfx \in \Til^n$ and $t \ges 0$.
\end{alist}
Moreover $k$ is unital if and only if equality holds in \tu{(ii)}
\tu{(}resp.\ \tu{(iii))}.
\item
 If $\ov$ is an operator system and $k$ is completely positive and
 contractive, then the following are equivalent\tu{:}
\begin{alist}
\item
$k$ is unital.
\item
Equality holds in \tu{(}b\tu{)}\tu{(}ii\tu{)}.
\item
Equality holds in \tu{(}b\tu{)}\tu{(}iii\tu{)}.
\item
$\mapassocsemigp^{x,x}$ is unital for all $x\in\Til$.
\end{alist}
\end{alist}
\end{propn}

\begin{proof}
Suppose that $\ov$ acts on $\init$. We start with three observations.

First note that, for $N \ges 1$, $u^1, \ldots, u^N \in \init$ and
$f_1, \ldots, f_N \in \Step^t_\Til$,
\begin{equation}\label{norm Schur Gram}
 \left\| \sum_{i=1}^N u^i \ot \w{f_i} \right\|^2 =
 \big\langle \bfu, \bigl[
I_\init \ot (\vp^{\bfx(0)}_{t_1-t_0} \schur \ldots \schur
\vp^{\bfx(n)}_{t_{n+1}-t_n}) \bigr] \bfu \big\rangle,
\end{equation}
where $\bfx (k) := \bff (t_k) \in \Til^N$ for $k=1, \ldots, n$,
whenever $\{0 = t_0 \les \cdots \les t_n =t\}$ contains the
discontinuities of each $f_i$.

Secondly,
\begin{equation}
 \label{kCP eqn}
\sum_{i,j=1}^N
 \big\langle u^i \ot \w{f_i}, k_t (a^i_j) u^j \ot \w{f_j}\big\rangle =
 \big\langle\bfu, k^\bff_t ([a^i_j]) \bfu \big\rangle
\end{equation}
where $k^\bff$ is the Schur-action map on $\Mat_N (\ov)$
from~\eqref{kf semigroup rep}.

Thirdly, the isometries defined by $F_{\bfx,t} := (I_\init \ot
W^\bfx_t) E(0)$, $\bfx \in \Til^n$ (in the notation~\eqref{Ef
notation}), satisfy
\begin{equation}\label{kCPC eqn}
F_{\bfx,t}^* (I_\init \ot \la \ot I_\Fock) F_{\bfx,t} = I_\init \ot
(\wxt \schur \la)
\end{equation}
for each $\la \in \Mat_n (\Comp)$.

\medskip
(a) If $k$ is CP then each $\mapassocsemigp^\bfx$ is CP by
Proposition~\ref{Wcp}(d). Conversely, suppose that each
$\mapassocsemigp^\bfx$ is positive. Then, by a reindexing, we see
that each $(\mapassocsemigp_t^{\bfx})^{(n)}$ is of the form
$\mapassocsemigp^{\bfx'}_t$, and so each $\mapassocsemigp^\bfx$ is
actually CP and therefore, by Proposition~\ref{kft maps}, each
$k^\bff_t$ is CP. By adaptedness, the density of $\init \aot
\Exps^t_\Til$ in $\init \ot \Fock^t$ and~\eqref{knN}, it follows
from~\eqref{kCP eqn} that $k$ is CP.

\medskip
(b) Suppose that the conditions of (a) hold and let $\Sigma$ be the
tensor flip $\ov \otm B(\Fock \ot \Comp^n) \to \ov \otm B(\Comp^n \ot
\Fock)$. If $k$ is also contractive then, by Proposition~\ref{Wcp}(d),
\[
\mapassocsemigp^\bfx_t (I_\init \ot \la) = F^*_{\bfx,t} \Sigma(k_t
(I_\init) \ot \la) F_{\bfx,t} \les F^*_{\bfx,t} (I_{\init} \ot \la
\ot I_\Fock) F_{\bfx,t}
\]
for $\la \in \Mat_n (\Comp)_+$, so (ii) holds by~\eqref{kCPC eqn}.

If (iii) holds then for any $\la \in \Mat_n (\Comp)_+$
\[
\mapassocsemigp^\bfx_t (I_\init \ot \la) = \mapassocsemigp^\bfx_t
(I_\init \ot \square_n) \schur \la \les I_\init \ot (\wxt \schur
\la)
\]
by Lemma~\ref{pos S prod}. Thus (ii) holds.

Finally suppose that (ii) holds and let $\xi = \sum^N_{i=1} u^i \ot
\w{f_i}$ where $\bff \in (\Step^t_\Til)^n$ has all its discontinuities
in $\{0 = t_0 \les \cdots \les t_{n+1} = t\}$. Then,
 by~\eqref{kCP eqn},~\eqref{kf semigroup rep} and~\eqref{norm Schur Gram}
\begin{align*}
\ip{\xi}{k_t (I_\init) \xi}
&= \big\langle \bfu, k^\bff_t (I_\init \ot \square_N)\bfu \big\rangle \\
&\les
 \big\langle \bfu, [I_\init \ot (\vp^{\bfx (0)}_{t_1 - t_0} \schur \cdots
\schur \vp^{\bfx (m)}_{t_{n+1} - t_n})] \bfu \big\rangle =
\norm{\xi}^2.
\end{align*}
Thus, by complete positivity, $k$ is a contraction cocycle.

Since (iii) is a special case of (ii), this proves the equivalences.

\medskip
(c)
 Tracing back through the proof of (b) confirms the equivalence
 of (i), (ii) and (iii).
  In case $n=1$, (iii) reads $\mapassocsemigp^{x,x}_t(I_\init) = I_\init$
  for all $x\in\Til$ and $t\in\Rplus$,
  so (iii) implies (iv).
  Suppose finally that (iv) holds.
  Then, for $x,y\in\Til$ and $t\in\Rplus$, setting
 \[
  A = (I_{\init\ot\Fock} - k_t(I_\init))^{1/2}, \quad
  X = AE_{\vp(x_{[0,t[})} \ \text{ and } \ Y = AE_{\vp(y_{[0,t[})},
 \]
 we have $X^*X = I_\init - \mapassocsemigp^{x,x}_t(I_\init) = 0$ so $X=0$,
 and so
 $E^{\vp(y_{[0,t[})} A^2 E_{\vp(x_{[0,t[})} = Y^*X = 0$.
 By the density of $\Exps_\Til$ it follows that $A^2=0$, in other
 words $k$ is unital and so (i) holds.
\end{proof}

If $f_1, \ldots, f_n$ are distinct vectors in a Hilbert space, $u_1,
\ldots, u_n$ are any vectors from another Hilbert space and $\sum
u_i \ot \w{f_i} = 0$, then $u_1 = \cdots = u_n =0$. This
straightforward extension of the well-known linear independence of
exponential vectors implies the following result, in which we adopt
the notations
\[
 \Exps_{\Til, t} :=
 \Lin \big\{\w{f}: f \in \Step_{\Til, t} \big\}
 \text{ where } \Step_{\Til, t} :=
 \big\{f \in \Step_\Til: \supp f \subset [0,t[ \big\}.
 \]
\begin{lemma}
 \label{SLM}
For a Hilbert space $\init$, a subset $\Til$ of $\noise$, a vector
space $U$ and a function $\psi: (\init \times \Step_{\Til, t})
\times (\init \times \Step_{\Til, t}) \To U$, if each function
$(u,v) \mapsto \psi \bigl( (u,f),(v,g) \bigr)$ is sesquilinear
$\init \times \init \To U$ then there is a unique sesquilinear map
$\Psi: (\init \aot \Exps_{\Til, t}) \times (\init \aot \Exps_{\Til,
t}) \To U$ satisfying
\[
\Psi (u \ot \w{f}, v \ot \w{g}) = \psi ((u,f),(v,g)).
\]
\end{lemma}

We now begin the task of reconstruction of cocycles from classes of
semigroups.

\begin{thm}\label{Q}
Let $\ov$ be an operator system and let $\mathcal{S} :=
\{\mapassocsemigp^{x,y}: x,y \in \Til\}$ be a family of semigroups
on $\ov$ indexed by a total subset $\Til$ of $\kil$ containing $0$.
Suppose that $\mathcal{S}$ satisfies the following conditions\tu{:}
\begin{alist}
\item
Each semigroup $\mapassocsemigp^\bfx$ is positive, and
\item
$\mapassocsemigp^\bfx_t (I_\init \ot \square_n) \les I_\init \ot
\wxt$ for all $n \ges 1$, $\bfx \in \Til^n$ and $t \ges 0$.
\end{alist}
Then there is a unique completely positive contraction cocycle $k$
on $\ov$ whose associated semigroups include $\mathcal{S}$.
Moreover, $k$ is unital if and only if each inequality in \tu{(b)}
is an equality.
\end{thm}

\begin{proof}
In this proof we identify elements of $\Step_{\Til, t}$ with their
restrictions to $[0,t[$.

Suppose that $\ov$ acts on $\init$. First note that the proof
(b\,iii $\Implies$ b\,ii) in Proposition~\ref{PP} applies to
$\mathcal{S}$, thus we may assume that this family satisfies
\begin{enumerate}
\item[(b)$'$]
$\mapassocsemigp^\bfx_t (I_\init \ot \la) \les I_\init \ot (\wxt
\schur \la)$ for all $n \ges 1$, $\bfx \in \Til^n$, $\la \in \Mat_n
(\Comp)_+$ and $t \ges 0$.
\end{enumerate}
For each $t > 0$ and $f,g \in \Step_{\Til, t}$ define a bounded map
$k(f,g,t)$ on $\ov$ by
\[
k(f,g,t) = \mapassocsemigp^{x_0,y_0}_{t_1-t_0} \circ \cdots \circ
\mapassocsemigp^{x_n,y_n}_{t_{n+1}-t_n}
\]
where $\{0 = t_0 \les \cdots \les t_{n+1} = t\}$ contains the
discontinuities of $f$ and $g$, and $(x_i,y_i) = (f(t_i),g(t_i))$.
That these maps are well-defined, that is independent of the choice
of subdivision of $[0,t[$, is a consequence of the semigroup
property of each $\mapassocsemigp^{x,y}$.

Now, for each $(u,v) \in \init \times \init$, the map
\[
\ov \To \Comp, \, \quad a \mapsto \big\langle u, k(f,g,t) (a) v
\big\rangle
\]
defines a bounded linear functional on $\ov$. Thus Lemma~\ref{SLM}
implies the existence of a sesquilinear map
 $(\init \aot \Exps_{\Til, t}) \times (\init \aot \Exps_{\Til, t})
 \To \ov^*$,
denoted $(\xi,\eta) \mapsto k_t [\xi,\eta]$, satisfying
\[
k_t [\xi,\eta] (a) = \sum_{i,j=1}^N \big\langle u^i, k(f_i,g_j,t)
(a) v^j \big\rangle
\]
for $\xi = \sum_{i=1}^N u^i \ot \w{f_i}$ and $\eta = \sum_{j=1}^N
v^j \ot \w{g_j}$ in $\init \aot \Exps_{\Til, t}$. In particular, if
$a \in \ov_+$ then letting $\bfx (k) = \bff (t_k) \in \Til^N$ for
$k=0, \ldots, n$, where $\{0 = t_0 \les \cdots \les t_{n+1} = t\}$
contains the discontinuities of $\bff$, and using the inequality $a
\ot \square_N \les \norm{a} I_\init \ot \square_N$,~\eqref{norm
Schur Gram} implies that
\begin{align} \label{form contract}
k_t [\xi, \xi] (a) & =
 \big\langle \bfu, \mapassocsemigp^{\bfx (0)}_{t_1-t_0} \circ \cdots
\circ \mapassocsemigp^{\bfx (n)}_{t_{n+1}-t_n} (a \ot \square_n) \bfu \big\rangle \\
&\les \norm{a} \big\langle \bfu, \bigl[ I_\init \ot
(\vp^{\bfx(0)}_{t_1-t_0} \schur \cdots \schur \vp^{\bfx
(n)}_{t_{n+1}-t_n}) \bigr] \bfu \big\rangle = \norm{a} \norm{\xi}^2.
\notag
\end{align}
Since $\ov = \Lin \ov_+$ this implies that, for any $a \in \ov$, the
quadratic form $\xi \mapsto k_t [\xi, \xi](a)$ is bounded. Therefore
there is a bounded operator $k(t,a)$ on $\init \ot \Fock_{[0,t[}$
satisfying $\ip{\xi}{k(t,a) \xi} = k_t [\xi, \xi] (a)$. Moreover,
by~\eqref{form contract},
\begin{equation}\label{pos contract}
k(t,a) \ges 0 \text{ for } a \in \ov_+.
\end{equation}
By the linearity of each $k_t [\xi, \xi]$, $k(t,a)$ is linear in $a$
and so $k_t (a) = k(t,a) \ot I_{[t,\infty [}$ defines an adapted
family of linear maps $k_t: \ov \To B(\init \ot \Fock)$.
By~\eqref{pos contract} each $k_t$ is positive and so also bounded
(by Proposition~\ref{A}). Since, for $f,g \in \Step_\Til$,
\[
k^{f,g}_t = k(f|_{[0,t[}, g|_{[0,t[}, t) =
\mapassocsemigp^{x_0-y_0}_{t_1-t_0} \circ \cdots \circ
\mapassocsemigp^{x_n-y_n}_{t_{n+1}-t_n},
\]
where $(x_l,y_l) := (f(t_l),g(t_l))$ for $l=1, \ldots, n$ and
$\{0=t_0 \les \cdots \les t_{n+1} = t\}$ contains the
discontinuities of both $f_{[0,t[}$ and $g_{[0,t[}$, $k$ is a
process on $\ov$ (Lemma~\ref{slice}); it is therefore a bounded
positive QS cocycle on $\ov$ whose associated semigroups include
$\mathcal{S}$, by Proposition~\ref{J}. Complete positivity and
contractivity of $k$ now follow from Proposition~\ref{PP}.
Uniqueness follows from Corollary~\ref{cor: uniqueness}. The last
part is contained in Proposition~\ref{PP}.
\end{proof}

In~\cite{version1} the characterisation of CP unital QS cocycles on
$B(\init)$ with one-dimensional $\noise$ is given in terms of a
single Schur-action CP semigroup on $\Mat_2 (B(\init))$, rather than
a family $\mathcal{S} = \{\mapassocsemigp^{x,y}: x,y \in \Til\}$ of
semigroups on $\ov$ as above. We next show two ways in which such a
global characterisation can be obtained for cocycles on operator
systems with multidimensional noise; the first requires a
separability assumption.
 Recall the definition of $\mapassocsemigp^\Til$ for a QS cocycle and total subset
 $\Til$ of $\noise$ containing $0$ (Proposition~\ref{Wcp}).

\begin{thm}
Let $\noise$ be separable, let $\Til$ be a countable total subset of
$\noise$ that contains $0$ and let $\mapassocsemigp$ be a semigroup
on $\Mat_\Til (\ov)_\bd$ for an operator system $\ov$. Then
$\mapassocsemigp$ is of the form $\mapassocsemigp^\Til$ for some
completely positive contraction cocycle if and only if
\begin{alist}
\item
$\mapassocsemigp$ has Schur-action,
\item
 $\mapassocsemigp$ is completely positive, and
\item
for some $\zeta = (\zeta^x) \in l^2 (\Til)$ with $\zeta^x \neq 0$ for
each $x \in \Til$ we have
\begin{equation}\label{Glob ineq}
\mapassocsemigp_t (I_\init \ot \La) \les I_\init \ot (\wTt \schur
\La) \text{ for all } t \ges 0
\end{equation}
where $\La = \dyad{\zeta}{\zeta} \in B(l^2(\Til)) =
\Mat_\Til(\Comp)_\bd$.
\end{alist}
In this case~\eqref{Glob ineq} holds for all $\La \in B(l^2(\Til))_+$.
\end{thm}

\begin{rem}
The statement above already illustrates one issue arising with the
passage to multidimensions: in general $\wTt := [\ip{\w{x_{[0,t[}}}
{\w{y_{[0,t[}}}]_{x,y \in \Til}$ will not define an element of
$B(l^2(\Til))$. However, it is a Schur multiplier, with the map $\La
\mapsto \wTt \schur \La$ being CP and unital on $B(l^2(\Til))$;
indeed
\[
\wTt \schur \La := F^*_{\Til,t} (I_\init \ot \La \ot I_\Fock) F_{\Til,t}
\]
where $F_{\Til,t} := (I_\init \ot W^\Til_t) E(0)$ (cf.\ the earlier
$F_{\bfx,t}$ notation).
\end{rem}

\begin{proof}
That~\eqref{Glob ineq} holds when $\mapassocsemigp$ is of the form
$\mapassocsemigp^\Til$, for some completely positive contraction
cocycle, follows by the same argument as in the proof of
Proposition~\ref{PP}. Assume, conversely, that $\mapassocsemigp$ has
Schur-action and~\eqref{Glob ineq} holds for $\La$ of the given
form. Then $\mapassocsemigp$ has components
$[\mapassocsemigp^{x,y}_t]_{x,y \in \Til}$ as defined
through~\eqref{cpt maps}. Let $n \ges 1$ and $\bfx \in \Til^n$, set
$\mapassocsemigp^\bfx_t = [\mapassocsemigp^{x_i,x_j}_t]: \Mat_n(\ov)
\To \Mat_n(\ov)$, and note that
\[
\mapassocsemigp^\bfx_t (A) = G_\bfx^* \mapassocsemigp^{(n)}_t
(G_\bfx A G^*_\bfx) G_\bfx
\]
where $G_\bfx = \diag [G_1 \cdots G_n]: \init^n \To (\init \ot
l^2(\Til))^n$ is the diagonal operator with $G_i = E_y$ for $y =
\de^{x_i}$. Thus $\mapassocsemigp^\bfx_t$ is a positive map.
Moreover, setting $\la^\bfx := (\zeta^{x_i}) \in \Comp^n$
\begin{align*}
\mapassocsemigp^\bfx_t (I_\init \ot \square_n) \schur
\dyad{\la^\bfx}{\la^\bfx} &=
\mapassocsemigp^\bfx_t (I_\init \ot \dyad{\la^\bfx}{\la^\bfx}) \\
&= G^*_\bfx \bigl(\mapassocsemigp_t (I_\init \ot
\dyad{\zeta}{\zeta}) \ot \square_n
\bigr) G_\bfx \\
&\les G^*_\bfx \bigl( I_\init \ot (\wTt \schur \dyad{\zeta}{\zeta})
\ot \square_n \bigr) G_\bfx \\
&= (I_\init \ot \wxt) \schur \dyad{\la^\bfx}{\la^\bfx}.
\end{align*}
Setting $\mu^\bfx := (1/\zeta^{x_i}) \in \Comp^n$, then
$\dyad{\mu^\bfx}{\mu^\bfx} \in B(\Comp^n)_+ = \Mat_n (\Comp)_+$ is
the Schur inverse of $\dyad{\la^\bfx}{\la^\bfx}$, and so
\[
\mapassocsemigp^\bfx_t (I_\init \ot \square_n) \les I_\init \ot \wxt
\]
by Lemma~\ref{pos S prod}. Thus Theorem~\ref{Q} applies giving the
existence of a cocycle $k$ with global semigroup $\mapassocsemigp$.
\end{proof}

\begin{rem}
If $\Til \subset \noise$ is uncountable then, for any $\La \in
B(l^2(\Til)) = \Mat_\Til (\Comp)_\bd$, only countably many components
of $\La$ in each row and column can be nonzero, and therefore many
finite submatrices will not be Schur-invertible.
%To extend the above
%to nonseparable noise dimension space $\kil$, $\mapassocsemigp_t(I_\init \ot
%\square)$ may instead be defined as a sesquilinear form with domain
%$\{\bfu \in \bigoplus_{x \in \Til} \init: u^x \neq 0 \text{ for only
%finitely many } x\}$.
% TOOOO recherche!
\end{rem}

Specialising to the case of unital cocycles we next give conditions
that guarantee the Schur-action of the global semigroup.

\begin{thm}\label{R}
Let $\mapassocsemigp$ be a semigroup on $\Mat_\Til (\ov)_\bd$ for an
operator system $\ov$ and a total subset $\Til$ of $\kil$ containing
$0$. Then $\mapassocsemigp$ is of the form $\mapassocsemigp^\Til$
for some completely positive unital QS cocycle on $\ov$ if and only
if $\mapassocsemigp$ is completely positive, contractive and
satisfies the normalisation conditions
\begin{equation}\label{normalisation}
\mapassocsemigp_t (I_\init \ot \dyad{\de^x}{\de^y}) =
e^{-t\chi(x,y)} I_\init \ot \dyad{\de^x}{\de^y}, \quad x,y \in \Til,
t \ges 0.
\end{equation}
\end{thm}

\begin{proof}
If $\mapassocsemigp=\mapassocsemigp^\Til$ for a CP unital QS cocycle
on $\ov$ then it is clear from the definition that $\mapassocsemigp$
is CP and satisfies
\[
\mapassocsemigp_t (I_\init \ot \dyad{\de^x}{\de^y}) =
\ip{\w{x_{[0,t[}}}{\w{y_{[0,t[}}} I_\init \ot \dyad{\de^x}{\de^y},
\]
which equals the right hand side of~\eqref{normalisation}.

Suppose conversely that $\mapassocsemigp$ is a CP semigroup on
$\Mat_\Til (\ov)$ satisfying~\eqref{normalisation}. Then
$\mapassocsemigp_t (I_\init \ot \dyad{\de^x}{\de^x})= I_\init \ot
\dyad{\de^x}{\de^x}$ for each $x \in \Til$ and so, by
Proposition~\ref{F}, $\mapassocsemigp_t$ has Schur-action.
Positivity of each $\mapassocsemigp^\bfx$ follows by a standard
reindexing argument, and each $\mapassocsemigp^\bfx$ is easily seen
to satisfy the conditions of Theorem~\ref{Q}(b) with equality, so
the proof is complete.
\end{proof}

\begin{rem}
The single normalisation condition cited in Theorem~21
of~\cite{version1} is not sufficient to guarantee that the semigroup
$\mapassocsemigp$ has Schur-action. For example, let $\init = \noise
= \Comp$, so that $\ov = \Comp$, and let $\Til = \{0,1\}$. Let
$\Phi$ be the completely positive map on $\Mat_2 (\Comp)$ defined by
$\Phi \bigl( \bigl[\begin{smallmatrix} a & b \\ c & d
\end{smallmatrix}\bigr] \bigr) = \bigl[\begin{smallmatrix} d & 0 \\
0 & a \end{smallmatrix}\bigr]$, and define $\Psi$ by $\Psi (A) =
\Phi (A) - A$. Then $\Psi$ is the generator of a unital completely
positive semigroup $\mapassocsemigp$ on $\Mat_2 (\Comp)$
(\cite{GKS}) satisfying
\[
\Psi \left(\begin{bmatrix} a & b \\ c & d \end{bmatrix}\right) =
\begin{bmatrix} d-a & -b \\ -c & a-d \end{bmatrix} \ \text{ and } \
\mapassocsemigp_t \left(\begin{bmatrix} 1 & 1 \\ 1 & 1
\end{bmatrix}\right) =
\begin{bmatrix} 1 & e^{-t/2} \\ e^{-t/2} & 1 \end{bmatrix}.
\]
This semigroup satisfies the normalisation condition
of~\cite{version1} (modified for the use of normalised exponential
vectors) but clearly does not have Schur-action and so cannot be the
global semigroup of a QS cocycle.
\end{rem}

\subsection*{Cocycles on a $C^*$-algebra}

Let $R$ and $T$ be self-adjoint operators on a Hilbert space $\init$.
Denote the set $\{S \in B(\init): R \les S \les T\}$ by $[R,T]$, and
for a $C^*$-algebra $\Al$ acting on $\init$ define
\[
[R,T]_\Al := [R,T] \cap \Al.
\]

\begin{thm}
Let $\Cil$ be a nonunital $C^*$-algebra acting nondegenerately, let
$\Til$ be a total subset of $\noise$ containing $0$, and let
$\mathcal{S} = \{\mapassocsemigp^{x,y}: x,y \in \Til\}$ be a family
of semigroups on $\Cil$. Then there is a completely positive
contraction cocycle on $\Cil$ whose associated semigroups include
$\mathcal{S}$ if and only if the family $\mathcal{S}$ satisfies
\begin{equation}\label{nonuni}
\mapassocsemigp^\bfx_t \bigl([0, I_\init \ot \la]_{\Mat_n
(\Cil)}\bigr) \subset \bigl[0, I_\init \ot (\la \schur
\wxt)\bigr]_{\Mat_n (\Cil)}
\end{equation}
for all $n \ges 1$, $\bfx \in \Til^n, \la \in \Mat_n (\Comp)_+$ and $t
\ges 0$.
\end{thm}

\begin{proof}
Let $\init$ be the Hilbert space on which $\Cil$ acts.

First let $k$ be a completely positive contraction cocycle on
$\Cil$. For each $t$ let $j_t$ be the extension of $k_t$
 to the unitisation of $\Cil$ in $B(\init)$ defined by
\[
j_t: \uCil \To B(\init \ot \Fock), \quad a +zI_\init \mapsto k_t (a)
+zI_{\init \ot \Fock}.
\]
Thus (by Proposition~\ref{Ua}) $j_t$ is CP and unital, and since
\[
j^{f,g}_t (a+zI_\init)= k^{f,g}_t (a) +
\ip{\w{f_{[0,t[}}}{\w{g_{[0,t[}}} zI_{\init \ot \Fock},
\]
it follows from Proposition~\ref{J} that $j$ is a QS cocycle on
$\uCil$. Let $\{\mathcal{R}^{x,y}: x,y \in \kil\}$ be its family of
associated semigroups. Then, for $\la \in \Mat_n (\Comp)_+$, $A \in
[0,I_\init \ot \la]_{\Mat_n(\Cil)}$ and $\bfx \in \Til^n$,
\[
0 \les \mapassocsemigp^\bfx_t (A) = \mathcal{R}^\bfx_t (A) \les
\mathcal{R}^\bfx_t (I_\init \ot \la) = I_\init \ot (\la \schur \wxt)
\]
by Proposition~\ref{PP}, as required.

The proof of the converse follows that of Theorem~\ref{Q}, except that
now the property $k_t[\xi,\xi](a) \les \norm{a} \norm{\xi}^2$ follows
from a careful iteration of~\eqref{nonuni}. Contractivity of the
resulting cocycle follows from Proposition~\ref{Ua}.
\end{proof}

\subsection*{Positive contraction operator cocycles}

We now apply the results of this section to the case of positive
contraction operator cocycles.
 This class has been studied in~\cite{jtl} under the assumption of
 \emph{Markov regularity} --- that is, norm-continuity of its
 expectation semigroup, and in~\cite{ccr},
 under the assumption of \ti{locality},
 or being a \emph{pure noise} cocycle
 --- that is,
 $X_t \in I_\init \ot B(\Fock_{[0,t[}) \ot I_{[t,\infty[}$
 for all $t\in\Rplus$.

\begin{thm}\label{W}
Let $\Til$ be a total subset of $\noise$ containing $0$ and let $S =
\{\opassocsemigp^{x,y}: x,y \in \Til\}$ be a family of semigroups on
a Hilbert space $\init$. Then there is a positive contraction
operator cocycle on $\init$ whose associated semigroups include $S$
if and only if $S$ satisfies the following conditions\tu{:}
\begin{alist}
\item\label{Qcomm}
 the family of operators
$\{\opassocsemigp^{x,y}_t: x,y \in \Til, t \ges 0\}$ is commutative,
 and
\item\label{Qineq}
$0 \les \opassocsemigp^\bfx_t \les I_\init \ot \wxt$ for all $n \ges
1$, $\bfx \in \Til^n$, and $t \ges 0$.
\end{alist}
\end{thm}

\begin{proof}
Define
 $\Cil := C^* \{\opassocsemigp^{x,y}_t: x,y \in \Til, t \ges 0\}$.

Let $X$ be a positive contraction cocycle on $\init$ with associated
semigroups $\{\opassocsemigp^{x,y}_t: x,y \in \kil\}$. By
self-adjointness $X$ is also a right cocycle and so, by the
semigroup decomposition~\eqref{Opj}, it follows that~\eqref{Qcomm}
holds. The completely bounded QS process on $\Cil $ defined by $k_t
(a) = X_t (a \ot I_\Fock) = (a \ot I_\Fock) X_t$ is a completely
positive contraction cocycle whose associated semigroups are given
by $\mapassocsemigp^{x,y}_t (a) = \opassocsemigp^{x,y}_t a$.
Therefore $\opassocsemigp^\bfx_t = \mapassocsemigp^\bfx_t (I_\init
\ot \square_n)$ for $n \ges 1$ and $\bfx \in \kil^n$, and
so~\eqref{Qineq} holds by Proposition~\ref{PP}.

Conversely, suppose that $S$ satisfies~\eqref{Qcomm}
and~\eqref{Qineq}. Then $\Cil$ is unital and abelian and
$\mapassocsemigp^{x,y}_t (a) := \opassocsemigp^{x,y}_t a$ defines a
family of semigroups $\mathcal{S} = \{\mapassocsemigp^{x,y}: x,y \in
\Til\}$ on $\Cil$ satisfying $\mapassocsemigp^\bfx_t (I_\init \ot
\square_n) \les I_\init \ot \wxt$. Since $\mapassocsemigp^\bfx_t (A)
= \opassocsemigp^\bfx_t \schur A$, it follows from the remark after
Lemma~\ref{pos S prod} that the semigroups $\mapassocsemigp^\bfx$
are positive. Thus Theorem~\ref{Q} implies that there is a
completely positive contraction cocycle $k$ on $\Cil$ whose
associated semigroups include $\mathcal{S}$; the associated
semigroups of the positive contraction operator cocycle on $\init$
defined by $X_t =k_t (I_\init)$ therefore include $S$.
\end{proof}

\subsection*{Completely contractive quantum stochastic cocycles}

We now apply Paulsen's $2 \times 2$ matrix trick to deal with
completely contractive cocycles on operator spaces.

Set
\begin{align*}
\Mat_n (\Comp)_{++} &:= \{\lambda \in \Mat_n (\Comp): \lambda \text{
is uniformly positive}\} \\
& \: =
 \Mat_n (\Comp)_+ \cap GL_n(\Comp).
\end{align*}

\begin{lemma}\label{pre S}
If $\la \in \Mat_n (\Comp)_{++}$ then
$\la \schur \wxt \in \Mat_n (\Comp)_{++}$ for each $\bfx$ in $\kil^n$.
\end{lemma}

\begin{proof}
If $\la \ges c I$ then, since $\wxt \ges 0$, we have $\la \schur \wxt
\ges c I \schur \wxt = c I$. The result follows.
\end{proof}

\begin{thm}\label{S}
Let $\ov$ be an operator space, let $\Til$ be a total subset of
$\noise$ containing $0$, and let $\mathcal{S} =
\{\mapassocsemigp^{x,y}: x,y \in \Til\}$ be a family of semigroups
on $\ov$. Then the following are equivalent\tu{:}
\begin{rlist}
\item\label{kCC}
There is a completely contractive QS cocycle on $\ov$ whose
associated semigroups include $\mathcal{S}$.
\item\label{Ptxpos}
Each semigroup $\wt{\mapassocsemigp}^\bfx$ \tu{(}constructed from
$\mathcal{S}$ by \eqref{u}\tu{)} is positive.
\item\label{PCCineq}
For all $n \ges 1$, $\bfx \in \Til^n$, $\la, \mu \in \Mat_n
(\Comp)_{++}$, $t \ges 0$ and $A \in \Mat_n(\ov)$,
\begin{equation}\label{LPM}
\norm{(\la \schur \wxt)^{-1/2} \mapassocsemigp^\bfx_t (\la^{1/2} A
\mu^{1/2}) (\mu \schur \wxt)^{-1/2}} \les \norm{A}.
\end{equation}
\end{rlist}
\end{thm}

\begin{proof}
Suppose that $\ov$ is an operator space in $B(\init; \init')$.

\smallskip
(\ref{kCC} $\Implies$ \ref{Ptxpos}): Let $k$ be a CC cocycle on
$\ov$. Then, by Corollary~\ref{B}, the cocycle $\wt{k}$ on
$\wt{\ov}$ is CP and unital, so by Proposition~\ref{PP} each
semigroup $\wt{\mapassocsemigp}^\bfx$ is positive.

\smallskip
(\ref{Ptxpos} $\Leftrightarrow$ \ref{PCCineq}): Fix $n \in \Nat$,
$\bfx \in \Til^n$ and $t \ges 0$, and let
$\wt{\mapassocsemigp}^\bfx$ be the semigroup constructed from
$\mapassocsemigp^\bfx$ as in~\eqref{op to map cocycles}. Under the
identification of $\Mat_n(\wt{\ov})$ given in~\eqref{a}, set
\[
M := \left\{ \begin{bmatrix} I_{\init'} \ot \la & A \\ A^* & I_\init
\ot \mu \end{bmatrix} \in \Mat_n (\wt{\ov}): \la \text{ and } \mu
\text{ are invertible} \right\}.
\]
By the characterisation~\eqref{PosMats} of nonnegative block matrices
\begin{multline*}
\begin{bmatrix} I_{\init'} \ot \la & A \\ A^* & I_\init \ot
\mu \end{bmatrix} \in M_+ \\
\Longleftrightarrow \la, \mu \in \Mat_n (\Comp)_{++} \text{ and } A
\in \Mat_n (\ov) \text{ satisfies } \norm{\la^{-1/2} A \mu^{-1/2}}
\les 1.
\end{multline*}

Since $M_+$ is dense in the closed subset $\Mat_n(\wt{\ov})_+$ of
$\Mat_n(\wt{\ov})$, it follows that $\wt{\mapassocsemigp}^{\bfx}_t$
is positive if and only if $\mapassocsemigp^\bfx_t$
satisfies~\eqref{LPM}.

\smallskip
(\ref{PCCineq} $\Implies$ \ref{kCC}): Assume that~\eqref{Ptxpos}
and~\eqref{PCCineq} hold. Consider the family of semigroups
$\wt{\mathcal{S}} := \{\wt{\mapassocsemigp}^{x,y}: x,y \in \Til\}$
on the operator system $\wt{\ov}$ constructed from the family of
semigroups $\mathcal{S}e$ by the prescription~\eqref{t}. Since each
semigroup $\wt{\mapassocsemigp}^\bfx$ is positive and satisfies
$\wt{\mapassocsemigp}^\bfx_t (I_{\init' \op \init} \ot \square_n) =
I_{\init' \op \init} \ot \wxt$, Theorem~\ref{Q} ensures the
existence of a CP unital cocycle $j$ on $\wt{\ov}$ whose associated
semigroups include $\wt{\mathcal{S}}$. Letting $\varepsilon$ be the
embedding $\ov \to \wt{\ov}$, $a \mapsto \bigl[\begin{smallmatrix} 0
& a \\ 0 & 0
\end{smallmatrix}\bigr]$ and $\pi$ be its left-inverse $\wt{\ov} \to
\ov$, $\bigl[\begin{smallmatrix} \lambda & a \\ b & \mu
\end{smallmatrix}\bigr] \mapsto a$, the prescription $k_t := \bigl(
\pi \otm \id_{B(\Fock)} \bigr) \circ j_t \circ \varepsilon$ ($t \ges
0$), defines a completely contractive QS cocycle on $\ov$ whose
associated semigroups include the family
\[
\bigl\{ \bigl(\pi \circ \wt{\mapassocsemigp}^{x,y}_t \circ
\varepsilon \bigr)_{t\ges 0}: x,y \in \Til \bigr\}
\]
which equals $\mathcal{S}$. Thus~\eqref{kCC} holds.
\end{proof}

\begin{rem}
The applicability of criterion~\eqref{PCCineq} derives from its manifest
stability under pointwise limits.
\end{rem}

\subsection*{Contraction operator quantum stochastic cocycles}

In this section we give our final characterisation. It is derived from
the characterisation of CC cocycles in part~\eqref{PCCineq} of
Theorem~\ref{S}, via the correspondence between CB cocycles on
$\ket{\init}$ and bounded operator cocycles on $\init$ given in
Proposition~\ref{op to map cocycles}.

\begin{thm}
Let $\Til$ be a total subset of $\noise$ containing $0$ and let
$S=\{\opassocsemigp^{x,y}: x,y \in \Til\}$ be a family of semigroups
on a Hilbert space $\init$. Then there is a left contraction
operator cocycle on $\init$ whose associated semigroups include $S$
if and only if for all $n \ges 1$, $\la, \mu \in
\Mat_n(\Comp)_{++}$, $\bfx \in \Til^n$, $t \ges 0$ and $A \in \Mat_n
(\ket{\init})$,
\[
\norm{(\la \schur \wxt)^{-1/2} \bigl( \opassocsemigp^\bfx_t \schur
(\la^{1/2} A \mu^{1/2}) \bigr) (\mu \schur \wxt)^{-1/2}} \les
\norm{A}.
\]
\end{thm}

\begin{proof}
This now follows from Proposition~\ref{op to map cocycles}, the
identities~\eqref{k12 semigps} and Theorem~\ref{S}.
\end{proof}

This theorem has led to new results on QS differential equations
with unbounded coefficients for contraction operator processes
(\cite{egs}). Moreover, a very recent infinitesimal analysis of
holomorphic contraction operator QS cocycles (\cite{LiS}), which
goes beyond the realm of quantum stochastic differential equations
as currently understood, is also underpinned by this
characterisation.

%%%%%%%%%%%%%%%%%%%%%%%%%%%%%%%%%%%%%%%%%%%%%%%%%
%\begin{comment}
%** I have commented out arXiv references to articles that have
%appeared in journals (for now) - one justification for keeping one in
%might be if it contains more than the journal version, like the
%Arveson article on domains. What do you think, Steve? **
%\end{comment}
%%%%%%%%%%%%%%%%%%%%%%%%%%%%%%%%%%%%%%%%%%%%%%%%%

 % deliberate mistake $
\end{document}